\documentclass[11pt, reqno]{amsart}
\usepackage{version}
\usepackage{amsmath,amsfonts,amssymb,amsthm}
\usepackage{mathtools}
\usepackage{mathrsfs}
\usepackage{nicefrac}
\usepackage[all]{xy}
\newtheorem{theorem}{Theorem}[section]
\newtheorem{lemma}[theorem]{Lemma}
\newtheorem{corollary}[theorem]{Corollary}
\newtheorem{proposition}[theorem]{Proposition}

\theoremstyle{definition}
\newtheorem{definition}[theorem]{Definition}

\newtheorem{remark}[theorem]{Remark}
\newtheorem{question}[theorem]{Question}

\newcommand{\la}{\langle}
\newcommand{\ra}{\rangle}

\newcommand{\B}{\mathbb{B}}
\newcommand{\C}{\mathbb{C}}
\newcommand{\D}{\mathbb{D}}
\renewcommand{\H}{\mathbb{H}}
\newcommand{\N}{\mathbb{N}}

\newcommand{\R}{\mathbb{R}}

\newcommand{\Z}{\mathbb{Z}}

\newcommand{\Aut}{{\rm Aut}}
\newcommand{\id}{{\sf id}}
\def\de{\partial}

\def\v{\varphi}

\renewcommand{\Im}{{\sf Im}\,}

\numberwithin{equation}{section}

\begin{document}
\title[Canonical models]{Canonical models for the forward and backward iteration of holomorphic maps}
\author[L. Arosio]{Leandro Arosio}
\address{L. Arosio: Dipartimento Di Matematica\\
Universit\`{a} di Roma \textquotedblleft Tor Vergata\textquotedblright\ \\
Via Della Ricerca Scientifica 1, 00133 \\
Roma, Italy} 
\email{arosio@mat.uniroma2.it}

\thanks{Supported by the ERC grant ``HEVO - Holomorphic Evolution Equations'' n. 277691}
\begin{abstract}
We prove the existence and the essential uniqueness of canonical models  for the forward (resp. backward) iteration of a holomorphic self-map $f$ of a cocompact Kobayashi hyperbolic complex manifold, such as the ball $\B^q$ or the polydisc $\Delta^q$. This is done performing a time-dependent conjugacy of the  dynamical system $(f^n)$, obtaining in this way a non-autonomous dynamical system admitting a relatively compact forward (resp. backward) orbit, and then proving the existence of a natural complex structure on a suitable quotient of the direct limit (resp.  subset of the inverse limit). As a corollary we prove the existence of a holomorphic solution with values in the upper half-plane of the Valiron equation for a holomorphic self-map of the unit ball.
\end{abstract}
\maketitle
\tableofcontents

\section*{Introduction}

In order to study the forward or backward iteration of  a holomorphic self-map $f\colon X\to X$  of a complex manifold, it is  natural  
to search for a semi-conjugacy of $f$ with some automorphism of a complex manifold. 
The first example of this approach is old as complex dynamics itself: if $\D\subset \C$ is the unit disc and  $f\colon \D\to \D$ is a holomorphic self-map such that $f(0)=0$ and   $0<|f'(0)|<1$,
then in 1884 K\"onigs proved \cite{Ko} that there exists a unique holomorphic mapping $h\colon \D\to \C$ solving the Schr\"oder equation 
$$h\circ f=f'(0) h,$$ and satisfying $h'(0)=1$.
Clearly $h$ gives a semi-conjugacy between $f$ and the automorphism $z\mapsto f'(0) z$ of $\C$. 
Notice that $\cup_{n\geq 0}f'(0)^{-n}h(\D)=\C$.

We call {\sl semi-model} for $f$
a triple $(\Lambda,h,\v)$, where $\Lambda$ is a complex manifold called the {\sl base space}, $h\colon X\to \Lambda$ is a 
holomorphic mapping  called the {\sl intertwining mapping} and $\v\colon \Lambda \to \Lambda$ is an automorphism, such that the following diagram commutes:
$$\xymatrix{X\ar[r]^{f}\ar[d]_{h}& X\ar[d]^{h}\\
\Lambda\ar[r]^{\v}& \Lambda,}$$
and $\Lambda=\bigcup_{n\geq0} \v^{-n}(h(X))$.
 A {\sl  model} for $f$ is a   semi-model  $(\Lambda,h,\v)$ such that  the   intertwining  mapping $h$  is univalent on an $f$-absorbing domain, that is, a domain $A$ such that $f(A)\subset A$ and such that every orbit of $f$ eventually lies in $A$.

There is a ``dual'' way of semi-conjugating $f$ with an automorphism: we call {\sl pre-model} for $f$
a triple $(\Lambda,h,\v)$, where $\Lambda$ is a complex manifold called the {\sl base space}, $h\colon \Lambda\to X$ is a 
holomorphic mapping called the {\sl intertwining mapping} and $\v\colon \Lambda \to \Lambda$ is an automorphism, such that the following diagram commutes:
$$\xymatrix{\Lambda\ar[r]^{\v}\ar[d]_{h}& \Lambda\ar[d]^{h}\\
X\ar[r]^{f}& X.}$$

We refer to \cite{ABmod, Ar, Ar2} for a brief history and recent developments in the theories of semi-models and pre-models. 
We recall that semi-models and pre-models, besides giving informations on the iteration of the self-map $f$, can also be fruitfully applied to the study of composition operators \cite{BS,Cowen3, jury,PC0} and of commuting self-maps \cite{Cowen2,B, BTV}.

We now need to recall some definitions and results for a holomorphic self-map $f$ of the unit disc $\D\subset \C$. A point $\zeta\in \partial \D$ is a {\sl boundary regular fixed point} if  $\angle \lim_{z\to \zeta}f(z)=\zeta, $ where $\angle \lim$ denotes the non-tangential limit, and if $$\lambda\coloneqq \liminf_{z\to \zeta} \frac{1-|f(z)|}{1-|z|}<+\infty.$$  The number $\lambda\in (0,+\infty)$ is called the {\sl dilation} of $f$ at $\zeta$. The point $\zeta$ is {\sl repelling} if $\lambda>1$.
The classical Denjoy--Wolff theorem states that if $f$ admits no fixed point $z\in \D$, then 
there exists a boundary regular fixed point $p\in \partial \D$ with dilation $\lambda \leq 1$ such that $(f^n)$ converges to the constant map $p$ uniformly on compact subsets.
The self-map $f$ is called {\sl hyperbolic} if $\lambda<1$.
We denote by $\H\subset \C$ the upper half-plane. 

We are interested in the following  examples of semi-models and pre-models in $\D$, given respectively by Valiron \cite{Va} and by Poggi-Corradini \cite{PC1}. Both examples can be seen as the solution of a generalized Schr\"oder equation at the boundary of the disc.

\begin{theorem}[Valiron]\label{blu}
Let $f\colon \D\to \D$ be a hyperbolic holomorphic self-map with dilation $\lambda<1$ at its Denjoy--Wolff point. Then there exists a model 
$(\H,h, z\mapsto \frac{1}{\lambda}z)$ for $f$.
\end{theorem}
\begin{theorem}[Poggi-Corradini]\label{bli}
Let $f\colon \D\to \D$ be a  holomorphic self-map and let $\zeta$ be a boundary repelling fixed point with dilation $\lambda>1$.
Then there exists a pre-model 
$(\H,h, z\mapsto \frac{1}{\lambda} z)$ for $f$.
\end{theorem}

A proof of the essential uniqueness of the intertwining mapping in Theorem \ref{blu} was given by Bracci--Poggi-Corradini \cite{BP}, and
Poggi-Corradini  \cite{PC1} proved that the intertwining mapping in Theorem \ref{bli} is essentially unique.

These two results were  generalized to the unit ball $\B^q\subset \C^q$ (for a definition of dilation, hyperbolic self-maps, Denjoy--Wolff point and boundary repelling points in the ball, see Sections \ref{theunitball} and \ref{theunitballback}).  Bracci--Gentili--Poggi-Corradini \cite{BGP} studied the case of a hyperbolic holomorphic self-map  $f\colon \B^q\to \B^q$ with dilation $\lambda<1$ at its Denjoy--Wolff point $p\in \partial \B^q$, and, assuming some regularity at $p$, they proved the existence of
 a one-dimensional semi-model $(\H,h, z\mapsto \frac{1}{\lambda}z)$ for $f$ (for other results about semi-models for hyperbolic self-maps, see \cite{BG,jury,bayart}). 
 
 Ostapyuk \cite{O} studied the case of a holomorphic self-map  $f\colon \B^q\to \B^q$ with a  boundary repelling fixed point $\zeta\in \partial \B^q$ with dilation $\lambda>1$, and,  assuming that $\zeta$  is isolated from other boundary repelling fixed points with dilation less or equal than $\lambda$, she proved  the existence of  a one-dimensional pre-model 
$(\H,h, z\mapsto \frac{1}{\lambda} z)$ for $f$.

Theorems \ref{blu} and \ref{bli} were generalized respectively  by Bracci and the author \cite{ABmod} and by the author \cite{Ar} to the case of  a univalent self-map $f\colon X\to X$ of a cocompact Kobayashi hyperbolic complex manifold (such as the unit ball $\B^q$ or the unit polydisc $\Delta^q$).  The approach used is geometric, much in the spirit of the work of Cowen \cite{Cowen} for the  forward iteration in the unit disc. 

We first consider the forward iteration case.
Let $k$ denote the Kobayashi distance. Notice that if $(z_n\coloneqq f^n(z_0))$ is a forward orbit, then for all fixed $m\geq 1$  the sequence $(k_{X}(z_n,z_{n+m}))_{n\geq 0}$ is non-increasing. The limit $s_m(z_0)\coloneqq \lim_{n\to\infty} k_{X}(z_n,z_{n+m})$ is called the {\sl forward $m$-step} at $z_0$. The {\sl divergence rate} of a self-map is a generalization introduced in  \cite{ABmod} of the dilation  at the Denjoy--Wolff point of a holomorphic self-map of $\B^q$. 
\begin{theorem}[A.--Bracci]\label{AF}
Let $X$ be Kobayashi hyperbolic and cocompact and let $f\colon X\to X$ be a univalent self-map. 
Then there exists an essentially unique  model $(\Omega,\sigma,\psi)$. Moreover, there  exists a holomorphic retract $Z$ of $X$, a  surjective holomorphic submersion $r\colon \Omega\to Z$, and an automorphism $\tau\colon Z\to Z$ with divergence rate   
\begin{equation}\label{cornflakes}
c(\tau)=c(f)=\lim_{m\to \infty}\frac{s_m(x)}{m},\quad x\in X,
\end{equation}
such that $(Z,r\circ \sigma,\tau)$ is a  semi-model for $f$, called a {\sl canonical Kobayashi hyperbolic semi-model}.

Moreover, the semi-model $(Z,r\circ \sigma,\tau)$ satisfies the following universal property. If $(\Lambda, h, \varphi)$ is a  semi-model for $f$ such that $ \Lambda$ is Kobayashi hyperbolic, then there exists a surjective holomorphic mapping $\eta\colon Z\to \Lambda$ such that
 the following diagram commutes:

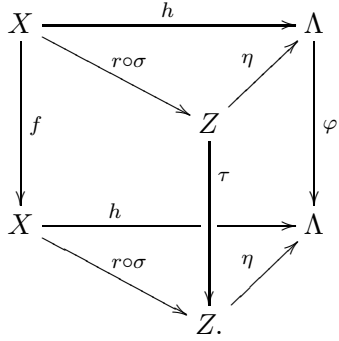
\SelectTips{xy}{12}
\[ \xymatrix{X \ar[rrr]^h\ar[rrd]^{r\circ \sigma}\ar[dd]^f &&& \Lambda \ar[dd]^\varphi\\
&& Z \ar[ru]^\eta \ar[dd]^(.25)\tau\\
X\ar'[rr]^h[rrr] \ar[rrd]^{r\circ \sigma} &&& \Lambda\\
&& Z \ar[ru]^\eta.}
\]

\end{theorem}
In particular, if $X=\B^q$ and $f$ is hyperbolic with dilation  $\lambda<1$ at its Denjoy--Wolff point, then   $Z$ is biholomorphic to a ball $\B^k$ with $1\leq k\leq q$, and the automorphism $\tau$ is hyperbolic with dilation  $\lambda$ at its Denjoy-Wolf point. 
As a corollary Theorem \ref{AF} yields the existence of a semi-model $(\H,\vartheta, z\mapsto \frac{1}{\lambda}z)$ for $f$, hence 
 $\vartheta\colon \B^q\to \H$ is a holomorphic solution of  the Valiron equation 
\begin{equation}\label{valironequation}
\vartheta\circ f=\frac{1}{\lambda}\vartheta.
\end{equation}

Now we recall the backward iteration case.
A {\sl backward orbit} is a sequence $\beta\coloneqq(y_n)$ in $X$ such that $f(y_{n+1})=y_n$ for all $n\geq 0$.
Notice that if $(y_n)$ is a backward orbit, then  for all fixed $m\geq 1$ the sequence $(k_{X}(y_n,y_{n+m}))_{n\geq 0}$ is non-decreasing. The limit $\sigma_m(\beta)\coloneqq \lim_{n\to\infty} k_{X}(y_n,y_{n+m})$   is called the {\sl backward $m$-step} of $\beta$. 
A backward orbit $\beta$ has {\sl bounded step} if $\sigma_1(\beta)<+\infty$.
\begin{theorem}[A.]\label{AB}
Let $X$ be Kobayashi hyperbolic and cocompact and let $f\colon X\to X$ be a univalent self-map.  Let $\beta\coloneqq(y_n)$ be a backward orbit for $f$ with bounded step.
Then there exists a holomorphic retract $Z$ of $X$, an injective  holomorphic immersion $\ell\colon Z\to X$, and an automorphism $\tau\colon Z\to Z$ with divergence rate   
\begin{equation}\label{cornflakes2}
c(\tau)=\lim_{m\to \infty}\frac{\sigma_m(\beta)}{m},
\end{equation}
such that $(Z,\ell,\tau)$ is a pre-model for $f$, called a {\sl canonical pre-model associated with $[y_n]$}.

Moreover $(Z,\ell,\tau)$ satisfies the following universal property.
If $(\Lambda,h,\v)$ is a pre-model for $f$ such that for some (and hence for any) $w\in \Lambda$, the non-decreasing sequence
$(k_X(h(\v^{-n}(w)),y_n))_{n\geq 0}$ is bounded, then 
there exists an injective holomorphic mapping $\eta\colon \Lambda\to Z$ such that
 the following diagram commutes:
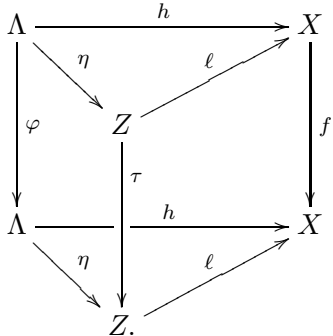
\SelectTips{xy}{12}
\[ \xymatrix{\Lambda\ar[rrr]^{h}\ar[rd]^\eta\ar[dd]^{\v} &&& X \ar[dd]^f\\
& Z \ar[rru]^\ell \ar[dd]^(.25)\tau
\\
\Lambda\ar'[r][rrr]^(.25){h} \ar[rd]^\eta &&& X\\
& Z \ar[rru]^\ell.}
\]

\end{theorem}
In particular, if $X=\B^q$ and the backward orbit $(y_n)$ converges to a boundary repelling fixed point $\zeta\in \partial \B^q$ with dilation $\lambda>1$, then $Z$ is biholomorphic to a ball $\B^k$ with $1\leq k\leq q$, and the automorphism $\tau$ is hyperbolic with dilation  $\mu\geq \lambda$ at its unique boundary repelling fixed point.

 In this paper we generalize Theorems \ref{AF} and \ref{AB} to non-necessarily univalent holomorphic self-maps $f\colon X\to X$, and then we apply our results to the case of the unit ball $\B^q$. Our proofs underline the strong duality between the forward case and the backward case.
 
In the first part of the paper  we prove Theorem \ref{principaleforward}, which generalizes Theorem  \ref{AF}.
Let $(\Omega,\Lambda_n\colon X\to \Omega)$ be the direct limit of the sequence $(f^n\colon X\to X)$.
Consider the equivalence relation $\sim$ on $\Omega$, where $[(x,n)],[(y,u)]\in \Omega$ are equivalent by $\sim$ if and only if  $$k_X(f^{m-n}(x),f^{m-u}(y))\stackrel{m\rightarrow\infty}\longrightarrow 0.$$
The bijective self-map $\psi\colon \Omega\to \Omega$ defined by $[(x,n)]\mapsto [(f(x),n)]$ satisfies $\psi\circ \Lambda_0=\Lambda_0\circ f$ and passes to the quotient inducing a bijective self-map
$\hat\psi \colon \Omega/_\sim\to  \Omega/_\sim$ satisfying 
$$\xymatrix{X\ar[r]^{f}\ar[d]_{\hat\Lambda_0}& X\ar[d]^{\hat\Lambda_0}\\
\Omega/_\sim\ar[r]^{\hat\psi}& \Omega/_\sim,}$$ 
where $\hat \Lambda_0\coloneqq \pi_\sim\circ \Lambda_0$.
A natural candidate for a canonical Kobayashi hyperbolic semi-model for $f$ would be the triple  $(\Omega/_\sim,\Lambda_0,\hat\psi)$.
Indeed, by the universal property of the direct limit, if $(\Lambda, h, \varphi)$ is a  semi-model for $f$ such that $ \Lambda$ is Kobayashi hyperbolic, then there exists a  mapping $\eta\colon \Omega/_\sim\to \Lambda$ which makes the following diagram commute:
\SelectTips{xy}{12}
\[ \xymatrix{X \ar[rrr]^h\ar[rrd]^{\hat\Lambda_0}\ar[dd]^f &&& \Lambda \ar[dd]^\varphi\\
&&  \Omega/_\sim \ar[ru]^\eta \ar[dd]^(.25){\hat\psi}\\
X\ar'[rr]^h[rrr] \ar[rrd]^{\hat\Lambda_0} &&& \Lambda\\
&&  \Omega/_\sim \ar[ru]^\eta.}
\]

We have to show that $\Omega/_\sim$ can be endowed with  a suitable complex structure.
If $f$ is univalent, then it follows from the proof of Theorem \ref{AF} that the direct limit $\Omega$ admits a natural complex structure which passes to the quotient to a complex structure on  $\Omega/_\sim$ 
 (see \cite{ABmod}). The  problem in the non-univalent case is that $\Omega$ may not admit a natural complex structure. 
 Rather surprisingly, even if $\Omega$ does not, the quotient set $\Omega/_\sim$ can always  be endowed with a complex structure which makes it biholomorphic to a holomorphic retract of $X$. We prove this by conjugating $(f^n)$ to a non-autonomous holomorphic forward dynamical system $(\tilde f_{n,m}\colon X\to X)_{m\geq n\geq 0}$ which admits a relatively compact forward orbit. This orbit is used to prove the existence of a 
 holomorphic retract $Z$ of $X$ and a family of holomorphic mappings $(\alpha_n\colon X\to Z)$ satisfying $$\alpha_m\circ \tilde f_{n,m}=\alpha_n,\quad \forall \ 0\leq n\leq m.$$
By the universal property of the direct limit there exists a mapping $\Phi\colon \Omega \to Z$ which induces a bijection $\hat\Phi\colon \Omega/_\sim\to Z$, which pulls back the desired complex structure to $\Omega/_\sim$.
Formula (\ref{cornflakes}) for the divergence rate of $\tau$ is a consequence of the fact that the  Kobayashi distance on $\Omega/_\sim$ admits a description in terms of the forward iteration of $f$.

In the second part of the paper, we consider the backward iteration of $f\colon X\to X$ and we prove Theorem \ref{principalebackward}, which generalizes Theorem  \ref{AB}.
Let $(\Theta,V_n\colon \Theta\to X)$ be the inverse limit  of the sequence $(f^n\colon X\to X)$.
Let $(y_n)$ be a backward orbit with bounded step and let $[y_n]\subset \Theta$ be the subset consisting  of the backward orbits $(z_n)\in \Theta$ such that the  non-decreasing sequence $(k_X(z_n,y_n))_{n\geq 0}$ is bounded. The bijective self-map   $\psi\colon \Theta\to \Theta$ defined by $(z_0, z_1, z_2,\dots)\mapsto [(f(z_0),z_0,z_1,\dots)]$ satisfies $\psi([y_n])=[y_n]$, and the following diagram commutes:
$$\xymatrix{[y_n]\ar[r]^{\psi|_{[y_n]}}\ar[d]_{V_0}& [y_n]\ar[d]^{V_0}\\
X\ar[r]^{f}& X.}$$

 A natural candidate for a canonical pre-model for $f$ associated with $[y_n]$  would be the triple $([y_n], V_0, \psi|_{[y_n]})$. Indeed, by the universal property of the inverse limit, if $(\Lambda, h, \varphi)$ is a  pre-model for $f$ such that for some (and hence for any) $w\in \Lambda$ the non-decreasing sequence
$(k_X(h(\v^{-n}(w)),y_n))_{n\geq 0}$ is bounded, then
  there exists a  mapping $\eta\colon \Lambda\to [y_n]$ which makes the following diagram commute:
  \SelectTips{xy}{12}
\[ \xymatrix{\Lambda\ar[rrr]^{h}\ar[rd]^\eta\ar[dd]^{\v} &&& X \ar[dd]^f\\
& [y_n] \ar[rru]^{V_0} \ar[dd]^(.25){\psi|_{[y_n]}}
\\
\Lambda\ar'[r][rrr]^(.25){h} \ar[rd]^\eta &&& X\\
& [y_n] \ar[rru]^{V_0}.}
\]

We have to show that $[y_n]$ can be endowed with  a suitable complex structure.
If $f$ is univalent, then $V_0\colon \Theta\to X$ is injective, and it follows from the proof of Theorem \ref{AB} that the image $V_0([y_n])$ is an injectively immersed complex submanifold of $X$ which is biholomorphic to a holomorphic retract of $X$.
In the non-univalent case the mapping $V_0\colon \Theta\to X$ is no longer injective, but the subset $[y_n]$ can however be endowed with a natural complex structure which makes it biholomorphic to a holomorphic retract of $X$.
We prove this by conjugating $(f^n)$ to a non-autonomous holomorphic backward dynamical system $(\tilde f_{n,m}\colon X\to X)_{m\geq n\geq 0}$ which admits a relatively compact  backward orbit. This orbit is used to prove the existence of a
 holomorphic retract $Z$ of $X$ and a family of holomorphic mappings $(\alpha_n\colon Z\to X)$ satisfying  $$\tilde f_{n,m}\circ \alpha_m=\alpha_n,\quad \forall \ 0\leq n\leq m.$$
 By the universal property of the inverse limit there exists an injective mapping $\Phi\colon   Z\to \Theta$, which pushes forward   the desired complex structure to its image  $\Phi(Z)=[y_n]$.
Formula (\ref{cornflakes2}) for the divergence rate of $\tau$ is a consequence of the fact that the  Kobayashi distance of $[y_n]$ admits a description in terms of the backward iteration of $f$.

\part{Forward iteration}
\section{Preliminaries}
\begin{definition}
Let $X$ be a complex manifold.
We call   {\sl forward (non-autonomous) holomorphic dynamical system} on $X$ any
family $(f_{n,m}\colon X\to X)_{m\geq n\geq 0}$ of holomorphic self-maps such that for all $m\geq u\geq  n\geq 0$, we have
$$f_{u,m}\circ f_{n,u}= f_{n,m}.$$ For all $n\geq 0$ we denote $f_{n,n+1}$ also by $f_n$.
A forward holomorphic dynamical system $(f_{n,m}\colon X\to X)_{m\geq n\geq 0}$ is called {\sl autonomous}
 if $f_{n}=f_{0}$ for all $n\geq 0$. Clearly in this case $f_{n,m}=f_0^{m-n}.$
\end{definition}
\begin{remark}
Any  family of holomorphic self-maps $(f_n\colon X\to X)_{n\geq 0}$ determines a   forward holomorphic dynamical system $(f_{n,m}\colon X\to X)$ in the following way: for all $n\geq 0$, set $f_{n,n}=\id$, and for all $m> n\geq 0$, set $$f_{n,m}=f_{m-1}\circ \cdots\circ f_n.$$ 
\end{remark}

\begin{definition}
Let $X$ be a complex manifold, and let $(f_{n,m}\colon X\to X)$ be a forward holomorphic dynamical system.
A {\sl direct limit} for $(f_{n,m})$ is a pair $(\Omega,\Lambda_n)$ where $\Omega$ is a set and
$(\Lambda_n\colon X\to \Omega)_{n\geq 0}$ is a family of  mappings such that 
 $$\Lambda_m\circ f_{n,m}=\Lambda_n,\quad \forall \ m\geq n\geq 0,$$
 satisfying the following universal property:
if $Q$ is a set and if $(g_n\colon X\to Q)$ is a family of  mappings satisfying
$$g_m\circ f_{n,m}=g_n,\quad \forall \ m\geq n\geq 0,$$
then there exists a  unique  mapping $\Gamma\colon \Omega\to Q$ such that
$$g_n=\Gamma\circ \Lambda_n,\quad \forall\ n\geq0.$$

\end{definition}
\begin{remark}
The direct limit is essentially unique, in the following sense.
Let $(\Omega,\Lambda_n)$ and $(Q, g_n)$ be two direct limits  for $(f_{n,m})$. Then there exists a bijective mapping 
$\Gamma\colon \Omega\to Q$ such that $$g_n=\Gamma\circ \Lambda_n,\quad \forall\ n\geq0.$$
\end{remark}
\begin{remark}
A direct limit for $(f_{n,m})$ is easily constructed.
We define an equivalence relation on the set $X\times \N$ in the following way: $(x,n)\simeq (y,m)$ if and only if
there exists $u\geq {\rm max}\{n,m\}$ such that $f_{n,u}(x)=f_{m,u}(y)$. We denote the equivalence class of $(x,n)$ by $[(x,n)]$, and we set $\Omega\coloneqq X\times \N/_{\simeq}$. We define a family of mappings $(\Lambda_n\colon X\to \Omega)_{n\geq 0}$ in the following way: for all $x\in X$ and $n\geq 0$, set $\Lambda_n(x)=[(x,n)]$. It is easy to see that $(\Omega,\Lambda_n)$ is a direct limit for $(f_{n,m})$.
\end{remark}

\begin{definition}
In what follows we will need  the following equivalence relation on $\Omega$:  $$[(x,n)]\sim [(y,u)]\quad\mbox{iff}\quad k_X(f_{n,m}(x),f_{u,m}(y))\stackrel{m\rightarrow\infty}\longrightarrow 0.$$ It is easy to see that this is well-defined.
We denote by $\pi_{\sim}\colon \Omega\to  \Omega/_\sim$ the projection to the quotient.
\end{definition}

We now introduce a modified version of the direct limit for $(f_{n,m})$  which is more suited for our needs.

\begin{definition}
Let $X$ be a complex manifold and let $(f_{n,m}\colon X\to X)$ be a forward  holomorphic dynamical system.
We call {\sl canonical Kobayashi hyperbolic direct limit} for $(f_{n,m})$ a pair $(Z,\alpha_n)$ where $Z$ is a Kobayashi hyperbolic  complex manifold and
$(\alpha_n\colon X\to Z)_{n\geq 0}$ is a family of holomorphic mappings such that 
 $$\alpha_m\circ f_{n,m}=\alpha_n,\quad \forall \ m\geq n\geq 0,$$
 which satisfies the following universal property:
if $Q$ is a Kobayashi hyperbolic complex manifold and if $(g_n\colon X\to Q)$ is a family of holomorphic mappings satisfying
$$g_m\circ f_{n,m}=g_n,\quad \forall \ m\geq n\geq 0,$$
then there exists a  unique holomorphic mapping $\Gamma\colon Z\to Q$ such that
$$g_n=\Gamma\circ \alpha_n,\quad \forall\ n\geq0.$$
\end{definition}

\begin{proposition}
The canonical Kobayashi hyperbolic direct limit is essentially unique, in the following sense.
Let $(Z,\alpha_n)$ and $(Q, g_n)$ be two canonical Kobayashi hyperbolic direct limits  for $(f_{n,m})$. Then there exists a biholomorphism 
$\Gamma\colon Z\to Q$ such that $$g_n=\Gamma\circ \alpha_n,\quad \forall\ n\geq0.$$

\end{proposition}

\begin{proof}
There exist holomorphic mappings $\Gamma\colon Z\to Q$ and $\Xi\colon Q\to Z$ such that for all $n\geq 0$, we have
$g_n=\Gamma\circ \alpha_n$ and $\alpha_n=\Xi\circ g_n$.
Thus the holomorphic mapping $\Xi\circ\Gamma\colon Z\to Z$ satisfies $$\Xi\circ\Gamma\circ\alpha_n=\alpha_n,\quad \forall\ n\geq0,$$
By the universal property of the canonical Kobayashi hyperbolic direct limit, this implies that $\Xi\circ\Gamma=\id_Z$. Similarly, we obtain $\Gamma\circ\Xi=\id_Q$.
\end{proof}

\section{Non-autonomous iteration}

Let $X$ be a taut complex manifold.
Let $(f_{n,m}\colon X\to X)_{m\geq n\geq 0}$ be a forward  holomorphic dynamical system, and assume that it admits 
a relatively compact forward orbit $(f_{0,m}(x_0))_{m\geq 0}$.

\begin{remark}\label{anche}
Let $K\subset X$ be a compact subset such that
$\{f_{0,m}(x_0)\}_{m\geq  0}\subset  K.$
It follows that,  for all fixed $n\geq 0$,  
\begin{equation}\label{converge}
f_{n,m}(K)\cap K\neq \varnothing\quad \forall\, m\geq  n.
\end{equation}
\end{remark}

The sequence of holomorphic self-maps $(f_{0,m}\colon X\to X)_{m\geq 0}$ is not compactly divergent by (\ref{converge}),  and since $X$ is taut, there exists a subsequence $(f_{0,m_{k_0}})_{k_0\geq 0}$ converging uniformly on compact subsets to a holomorphic self-map
$\alpha_0\colon X\to X$. The sequence of holomorphic self-maps $(f_{1,m_{k_0}}\colon X\to X)_{k_0\geq 0}$
is not compactly divergent  by (\ref{converge}), and since $X$ is taut, there exists a subsequence $(f_{1,m_{k_1}})_{k_1\geq 0}$ converging to a holomorphic self-map $\alpha_1\colon X\to X$. Iterating this  procedure we obtain a family of holomorphic self-maps
$(\alpha_n\colon X\to X)$ satisfying for all $m\geq n\geq 0$,
\begin{equation}\label{hermione}
\alpha_m\circ f_{n,m}=\alpha_n.
\end{equation}
Notice that for all $n\geq 0$ we have
\begin{equation}\label{hagrid}
\alpha_n(K)\cap K\neq \varnothing.
\end{equation}
Let now $(m_k)_{k\geq 0}$ be a sequence of integers which for all $j\geq 0$ is eventually a subsequence of $(m_{k_j})_{k_j\geq 0}$ (such a sequence exists by a classical diagonal argument).

The sequence of holomorphic self-maps $(\alpha_{m_k}\colon X\to X)_{k\geq 0}$
 is not compactly divergent by (\ref{hagrid}), and since  $X$ is taut, there exists a subsequence $(\alpha_{m_h})_{h\geq 0}$
converging uniformly on compact subsets to a holomorphic self-map $\alpha\colon X\to X$.

\begin{proposition}\label{x1}
The holomorphic self-map $\alpha\colon X\to X$ is a holomorphic retraction, and for all $n\geq 0$,
\begin{equation}\label{potter}
\alpha\circ \alpha_n=\alpha_n.
\end{equation}
\end{proposition}
\begin{proof}
Fix $n\geq 0$ and $x\in X$. Then for all $h\geq 0$ such that  $m_h\geq n$, we have 
$$\alpha_n(x)=\alpha_{m_h}( f_{n,m_h}(x))\stackrel{h\rightarrow \infty}\longrightarrow \alpha( \alpha_n(x)).$$
Thus we have, for all $h\geq 0$,  $$\alpha( \alpha_{m_h}(x))=\alpha_{m_h}(x).$$
When $h\rightarrow \infty$, the left-hand side converges to $\alpha(\alpha(x))$, while the right-hand side converges to $\alpha(x)$.
\end{proof}
\begin{remark}
The image $\alpha(X)$ is a closed complex submanifold of $X$ (see e.g. \cite[Lemma 2.1.28]{A}).
\end{remark}
\begin{definition}
We denote $\alpha(X)$ by $Z$.
\end{definition}

\begin{remark}
By (\ref{potter}), we have $\alpha_n(X)\subset Z$ for all $n\geq 0$, and by (\ref{hermione}) we have that
$$\alpha_n(X)\subset \alpha_m(X)$$ for all $m\geq n\geq 0$.
\end{remark}

Let $(\Omega,\Lambda_n)$ be the direct limit of the directed system $(X,f_{n,m})$. By the universal property of the direct limit,
there exists a  mapping $\Psi\colon \Omega\to Z$ such that for all $n\geq 0$, $$\alpha_n=\Psi\circ \Lambda_n.$$
The mapping $\Psi$ is defined in the following way: if $[(x,n)]\in \Omega$, then $\Psi([(x,n)])=\alpha_n(x)$.

\begin{proposition}\label{x2}
The mapping  $\Psi\colon \Omega\to Z$ is surjective, and $\Psi([(x,n)])=\Psi([(y,u)])$ if and only if $[(x,n)]\sim [(y,u)]$.

\end{proposition}
\begin{proof}
 Since $\alpha$ is a retraction, we have $\alpha(z)=z$ for all $z\in Z$, that is, $\alpha_{m_h}(z)\stackrel{h\rightarrow\infty}\longrightarrow z$ for all $z\in Z$.  Consider the sequence of holomorphic mappings $(\alpha_{m_h}|_Z\colon Z\to Z)$. This sequence converges uniformly on compact subsets to $\id_Z$, and thus it is eventually injective on compact subsets of $Z$. Fix $z\in Z$ and let $U$ be a neighborhood of $z$ in $Z$ such that $(\alpha_{m_h}|_U\colon U\to Z)$ is eventually injective. Then the image $\alpha_{m_h}|_U$ eventually contains $z$ (see e.g. \cite[Corollary 3.2]{abstract}). Hence we obtain that $\Psi\colon \Omega\to Z$ is surjective.

Assume now that $[(x,n)]\sim [(y,u)]$.
For all $m\geq {\rm max}\{n,u\}$, we have that $\Psi([(x,n)])=\alpha_m(f_{n,m}(x))$, and $\Psi([(y,u)])=\alpha_m(f_{u,m}(y))$.
We have $$k_X(\Psi([(x,n)]),\Psi([(y,u)]))\leq k_X(f_{n,m}(x),f_{u,m}(y))\stackrel{m\rightarrow\infty}\longrightarrow 0,$$
which implies  $\Psi([(x,n)])=\Psi([(y,u)])$.
 
Conversely, assume that $\Psi([(x,n)])=\Psi([(y,u)])$. 
It follows that  $$\lim_{h\to\infty}f_{n,m_h}(x)=\lim_{h\to\infty}f_{u,m_h}(y),$$ and thus $\lim_{h\to\infty}k_X(f_{n,m_h}(x),f_{u,m_h}(y))=0.$ Since the sequence  $(k_X(f_{n,m}(x),f_{u,m}(y)))_{m\geq {\rm max}\{n,u\}}$ is non-increasing, we have $[(x,n)]\sim [(y,u)]$.
\end{proof}
\begin{remark}\label{iphone}
It follows from Proposition \ref{x2} that $\bigcup_{n\geq 0}\alpha_n(X)=Z$, and that $\Psi$ induces a bijection $\hat\Psi\colon \Omega/_\sim\to Z$.
\end{remark}

\begin{proposition}\label{x3}
The pair $(Z,\alpha_n)$ is a canonical Kobayashi hyperbolic direct limit for $(f_{n,m})$.
\end{proposition}

\begin{proof}
First of all, $Z$ is Kobayashi hyperbolic since it is  a submanifold of $X$.
Let  $Q$ be a Kobayashi hyperbolic complex manifold and let $(g_n\colon X\to Q)$ be a family of holomorphic mappings satisfying
$$g_m\circ f_{n,m}=g_n,\quad \forall \ m\geq n\geq 0.$$
By the universal property of the direct limit, there exists a unique mapping $\Phi\colon \Omega\to Q$ such that 
$$g_n=\Phi\circ \Lambda_n,\quad \forall\ n\geq0.$$
The mapping $\Phi$ is defined in the following way: if $[(x,n)]\in \Omega$, then $\Phi([(x,n)])=g_n(x)$.
We claim that  $$[(x,n)]\sim [(y,u)]\Longrightarrow \Phi([(x,n)])=\Phi [(y,u)].$$
Indeed, if $[(x,n)]\sim [(y,u)]$, then
for all $m\geq {\rm max}\{n,u\}$, we have that $\Phi([(x,n)])=g_m(f_{n,m}(x))$, and $\Phi([(y,u)])=g_m(f_{u,m}(y))$.
We have $$k_X(\Phi([(x,n)]),\Phi([(y,u)]))\leq k_X(f_{n,m}(x),f_{u,m}(y))\stackrel{m\rightarrow\infty}\longrightarrow 0.$$

Thus there exists a unique mapping $\hat \Phi\colon \Omega/_\sim\to Q$ such that $\hat \Phi\circ \pi_{\sim}=\Phi$.

Set $$\Gamma\coloneqq \hat\Phi\circ {\hat\Psi}^{-1}\colon Z\to Q.$$
For all $n\geq 0$, $$\Gamma\circ \alpha_n=\Gamma\circ \Psi\circ \Lambda_n= \hat\Phi\circ\pi_\sim\circ \Lambda_n= \Phi\circ \Lambda_n=g_n.$$
The uniqueness of the mapping $\Gamma$ follows easily from the uniqueness of  the mappings $\Phi$ and  $\hat \Phi$.
The mapping $\Gamma $ acts in the following way: if $z\in Z$, then  
there exists $x\in X$ and $n\geq 0$ such that $\alpha_n(x)=z$, and then $\Gamma(z)=g_n(x).$
 
We now prove that $\Gamma$ is holomorphic. 
Let $z\in Z$, and let $x\in X$ and $n\geq 0$  such that $\alpha_n(x)=z$. 
Since $\alpha$ has maximal rank at $z$, there exists a neighborhood $V$ of $z$ in $X$ such that, for $m$ large enough,
$\alpha_m$ has maximal rank at every point $y\in V$.
Since the sequence $(f_{n,m_{k_n}}(x))_{k_n\geq 0}$ converges to $\alpha_n(x)=z$ as $k_n\to\infty$, it is eventually contained in $V$.
 Hence there exists $m'\geq 0$ such that $w\coloneqq f_{n,m'}(x)\in V$ and $\alpha_{m'}$ has maximal rank at $w$.
 Thus there exists an open neighborhood $U\subset Z$ of $z$ and a holomorphic function $\sigma\colon U\to X$ such that 
 $$\alpha_{m'}\circ\sigma=\id_U.$$
 Then, for all $y\in U$,
  $$\Gamma(y)=\Gamma(\alpha_{m'}(\sigma(y)))=g_{m'}(\sigma(y)),$$
 which means that $\Gamma$ is holomorphic in $U$.
\end{proof}

We denote by $\kappa$ the Kobayashi--Royden metric.
\begin{proposition}\label{x4}
For all $n\geq 0$,
\begin{equation}\label{ulula}
\lim_{m\to\infty}  f_{n,m}^*\, k_X=\alpha_n^* \, k_Z,
\end{equation}
and
\begin{equation}\label{ululi}
\lim_{m\to\infty} f_{n,m}^*\, \kappa_X=\alpha_n^* \, \kappa_Z.
\end{equation}
\end{proposition}

\begin{proof}
Let $x,y\in X$, and fix $n\geq 0$. We have that $$\lim_{k_n\to\infty} k_X(f_{n,m_{k_n}}(x),f_{n,m_{k_n}}(y))=k_X(\alpha_n(x),\alpha_n(y))=k_Z(\alpha_n(x),\alpha_n(y)),$$
where the last identity follows from the fact that $\alpha_n(x),\alpha_n(y)\in Z$ and $Z$ is a holomorphic retract.
Then (\ref{ulula}) follows since the sequence $(k_X(f_{n,m}(x),f_{n,m}(y)))_{m\geq n}$ is  non-increasing.

The  proof of  (\ref{ululi}) is similar.
\end{proof}


\begin{definition}
Let $X$ be a Kobayashi hyperbolic complex manifold. We say that $X$ is {\sl cocompact} if $X/{\rm aut}(X)$ is compact. 
\end{definition}
Notice that this implies that $X$ is complete Kobayashi  hyperbolic \cite[Lemma 2.1]{FS}. 

\begin{theorem}\label{cocononaut}
Let $X$ be a cocompact Kobayashi hyperbolic complex manifold, and let $(f_{n,m}\colon X\to X)_{m\geq n\geq 0}$ be a forward  holomorphic dynamical system. 
Then there exists a canonical Kobayashi hyperbolic direct limit $(Z,\alpha_n)$  for $(f_{n,m})$, where $Z$ is a holomorphic retract of $X$.
Moreover,
\begin{equation}\label{edward}
Z=\bigcup_{n\geq 0}\alpha_n(X),
\end{equation}
and 
\begin{equation}\label{timburton}
\lim_{m\to\infty}  f_{n,m}^*\, k_X=\alpha_n^* \, k_Z,\quad
\lim_{m\to\infty} f_{n,m}^*\, \kappa_X=\alpha_n^* \, \kappa_Z.
\end{equation} 
\end{theorem}
\begin{proof}

 Let $K\subset X$ be a compact subset  such that $X=\Aut(X)\cdot K$.
Let $x_0\in X$. For all $n\geq 0$, let $h_n\in\Aut(X)$ be such that $h_n(f_{0,n}(x_0))\in K$.
For all $m\geq n\geq 0$ set ${\tilde f}_{n,m}\coloneqq h_m\circ f_{n,m}\circ h_n^{-1}$. It is easy to see that
$({\tilde f}_{n,m}\colon X\to X)$ is a  forward  holomorphic dynamical system such that
\begin{equation}
\{{\tilde f}_{0,m}(h_0(x_0))\}_{m\geq  0}\subset  K.
\end{equation}
We can now apply Proposition \ref{x3} to $({\tilde f}_{n,m}\colon X\to X)$, obtaining a canonical Kobayashi hyperbolic direct limit $(Z,\tilde\alpha_n)$ for $({\tilde f}_{n,m})$, where $Z$ is a holomorphic retract of $X$.
For all $n\geq0$ set $\alpha_n\coloneqq \tilde\alpha_n\circ h_n$. Clearly $$\alpha_m\circ f_{n,m}=\alpha_n,\quad \forall \ m\geq n\geq 0.$$
Let $Q$ be a Kobayashi hyperbolic manifold and let $(g_n\colon X\to Q)$ be a family of holomorphic mappings satisfying
$$g_m\circ f_{n,m}=g_n,\quad \forall \ m\geq n\geq 0.$$
For all $n\geq0$ set $\tilde g_n\coloneqq g_n\circ h_n^{-1}$.
Then for all $m\geq n\geq 0$, $$\tilde g_m\circ {\tilde f}_{n,m}= g_m\circ h_m^{-1}\circ {\tilde f}_{n,m}=g_m\circ f_{n,m}\circ h_n^{-1}=g_n\circ h_n^{-1}=\tilde g_n.$$
By the universal property of the canonical Kobayashi hyperbolic direct limit applied to  $(Z,\tilde\alpha_n)$
we obtain a holomorphic mapping $\Gamma\colon Z\to Q$ such that
$$\tilde g_n=\Gamma\circ \tilde\alpha_n,\quad \forall\ n\geq0.$$
Hence $g_n=\Gamma\circ \alpha_n $ for all $n\geq0$.

Remark \ref{iphone} yields (\ref{edward}).
Finally, (\ref{timburton}) follows from Proposition \ref{x4} since for all $n\geq 0$ the automorphism $h_n\colon X\to X$ is an isometry for $k_X$ and $\kappa_X$.
\end{proof}

\begin{remark}
Let $(\Omega,\Lambda_n)$ be the direct limit of the directed system $(X,f_{n,m})$. Let $(Z,\alpha_n)$ be the canonical Kobayashi hyperbolic direct limit given by Theorem \ref{cocononaut}. By the universal property of the direct limit,
there exists a  mapping $\Psi\colon \Omega\to Z$ such that  $\alpha_n=\Psi\circ \Lambda_n$ for all $n\geq 0$.
It is easy to see that $\Psi$ is surjective and induces a bijection $\hat\Psi\colon \Omega/_\sim\to Z$ such that 
$$\alpha_n=\hat\Psi\circ \pi_\sim\circ \Lambda_n, \quad \forall \ n\geq 0.$$

\end{remark}

\section{Autonomous iteration}
\begin{definition}
Let $X$ be a complex manifold and let $f\colon X\to X$ be a holomorphic self-map.
Let $x\in X$, and let $m\geq 0$. The {\sl $m$-step} $s_m(x)$ of $f$ at $x$ is the limit
$$s_m(x)=\lim_{n\to\infty}k_X(f^n(x), f^{n+m}(x)).$$ Such a limit exists since the sequence $(k_X(f^n(x), f^{n+m}(x))_{n\geq 0}$ is non-increasing.
The {\sl divergence rate} $c(f)$ of $f$ is the limit $$c(f)=\lim_{m\to\infty}\frac{k_X(f^m(x),x)}{m}.$$ It is shown in \cite{ABmod} that such a limit exists, does not depend on $x\in X$ and equals $\inf_{m\in \N}\frac{k_X(f^m(x),x)}{m}$.
\end{definition}
\begin{definition}\label{semimodel}
Let $X$ be a complex manifold and let $f\colon X\to X$ be a holomorphic self-map.
A {\sl semi-model} for $f$ is a triple  $(\Lambda,h,\v)$  where
$\Lambda$ is a complex manifold, $h\colon X\to \Lambda$ is a holomorphic mapping, and $\v\colon \Omega\to\Omega$ is an automorphism such that
\begin{equation}
h\circ f=\v\circ h,
\end{equation}
and 
\begin{equation}\label{due}
\bigcup_{n\geq 0} \v^{-n}(h(X))=\Lambda.
\end{equation}

We call the manifold $\Lambda$ the {\sl base space} and the mapping $h$ the {\sl intertwining mapping}.

Let $(Z,\ell,\tau)$ and   $(\Lambda, h, \v)$ be two semi-models for $f$. A {\sl morphism of semi-models} 
$\hat\eta\colon (Z,\ell,\tau)\to (\Lambda, h, \v)$ is given by a  holomorphic map $\eta: Z\to \Lambda$ 
such that  the following diagram commutes:
\SelectTips{xy}{12}
\[ \xymatrix{X \ar[rrr]^h\ar[rrd]^\ell\ar[dd]^f &&& \Lambda \ar[dd]^\varphi\\
&& Z \ar[ru]^\eta \ar[dd]^(.25)\tau\\
X\ar'[rr]^h[rrr] \ar[rrd]^\ell &&& \Lambda\\
&& Z \ar[ru]^\eta.}
\]
If the mapping $\eta \colon Z\to \Lambda$ is a biholomorphism, then we say that $\hat\eta\colon  (Z,\ell,\tau)\to (\Lambda, h, \v)$ is an {\sl isomorphism of semi-models}. Notice that then $\eta^{-1}\colon \Lambda\to Z$ induces a morphism $ {\hat\eta}^{-1}\colon  (\Lambda, h, \v)\to  (Z,\ell,\tau). $
\end{definition}

\begin{remark}\label{help}
It is shown in \cite[Lemmas 3.6 and 3.7]{ABmod} that if $ (Z,\ell,\tau), (\Lambda, h, \v)$ are semi-models for $f$, then there exists at most one morphism $\hat\eta\colon (Z,\ell,\tau)\to (\Lambda, h, \v)$, and that the holomorphic map $\eta: Z\to \Lambda$ is surjective.

\end{remark}

\begin{definition} 
Let $X$ be a complex manifold and let $f\colon X\to X$ be a holomorphic self-map. Let $(Z, \ell,\tau)$ be a semi-model for $f$  whose base space $Z$ is Kobayashi hyperbolic. We say that  $(Z, \ell,\tau)$ is a {\sl canonical Kobayashi hyperbolic semi-model} for $f$  if for any semi-model
$(\Lambda, h,\varphi )$ for $f$ such that the base space $\Lambda$ is Kobayashi hyperbolic,  there exists a  morphism of semi-models $\hat\eta\colon (Z, \ell,\tau)\to (\Lambda, h,\varphi )$ (which is necessarily unique by Remark \ref{help}).
 \end{definition}
 \begin{remark}\label{tango}
If $(Z, \ell,\tau)$ and $(\Lambda, h,\varphi )$ are two canonical Kobayashi hyperbolic semi-models for $f$, then they are isomorphic.
\end{remark}

\begin{theorem}\label{principaleforward}
 Let $X$ be a cocompact  Kobayashi hyperbolic complex manifold, and let $f\colon X\to X$ be a holomorphic self-map. Then there exists a  canonical Kobayashi hyperbolic semi-model $(Z,\ell,\tau)$ for $f$, where $Z$ is a holomorphic retract of $X$.
Moreover, the following holds:
\begin{enumerate}
\item  if $\alpha_n\coloneqq\tau^{-n}\circ \ell$ for all $n\geq 0$, then
$$\lim_{m\to\infty}  (f^m)^*\, k_X=\alpha_n^* \, k_Z,\quad
\lim_{m\to\infty} (f^{m})^*\, \kappa_X=\alpha_n^* \, \kappa_Z,$$
\item the divergence rate of $\tau$ satisfies $$c(\tau)=c(f)=\lim_{m\to\infty}\frac{s_m(x)}{m}=\inf_{m\in\N}\frac{s_m(x)}{m}.$$
\end{enumerate}
\end{theorem}
\begin{proof}
Let $(f_{n,m}\colon X\to X)$ be the autonomous dynamical system defined by $f_{n,m}=f^{m-n}$. By Theorem \ref{cocononaut}, there exist
 a holomorphic retract $Z$ of $X$ and a family of holomorphic mappings
$(\alpha_n\colon X\to Z)$ such that the pair $(Z,\alpha_n)$ is a canonical Kobayashi hyperbolic direct limit for $(f_{n,m})$.
The sequence of holomorphic mappings $(\beta_n\coloneqq\alpha_n\circ f\colon X\to Z)$ satisfies, for all $m\geq n\geq 0$,
$$\beta_m\circ f_{n,m}=\alpha_m\circ f\circ f^{m-n}=\alpha_n\circ f=\beta_n.$$
By the universal property of the canonical Kobayashi hyperbolic direct limit there exists a holomorphic self-map $\tau\colon Z\to Z$ such that for all $n\geq 0$, 
 $$\tau\circ \alpha_n=\alpha_n\circ f.$$
We claim that $\tau$ is a holomorphic automorphism.
For all $n\geq 0$, set $\gamma_n\coloneqq \alpha_{n+1}$. For all $m\geq n\geq 0$,
$$\gamma_m\circ f_{n,m}=\alpha_{m+1}\circ f^{m-n}=\alpha_{n+1}=\gamma_n.$$
Thus there exists a holomorphic self-map $\delta\colon Z\to Z$ such that $\delta\circ \alpha_n=\alpha_{n+1}$ for all $n\geq 0$.
For all $n\geq 0$ we have $$\tau\circ\delta\circ \alpha_n=\tau\circ \alpha_{n+1}=\alpha_n,$$ and
$$\delta\circ \tau\circ \alpha_n=\delta\circ \alpha_n\circ f=\alpha_{n+1}\circ f=\alpha_n.$$ By the universal property of the canonical Kobayashi hyperbolic direct limit we have that $\tau$ is a holomorphic automorphism and $\delta=\tau^{-1}$. Since for all $n\geq 0$, $$\tau^n\circ \alpha_n=\alpha_n\circ f^n=\alpha_0,$$ it follows that $\alpha_n=\tau^{-n}\circ \alpha_0$.

Set $\ell\coloneqq \alpha_0$. We claim that the triple $(Z,\ell, \tau)$ is a canonical Kobayashi hyperbolic semi-model for $f$. 
Indeed, let $(\Lambda, h,\varphi )$ be a semi-model for $f$ such that the base space $\Lambda$ is Kobayashi hyperbolic. For all $n\geq 0$, let $\lambda_n\coloneqq  \v^{-n}\circ h$. Then by the universal property of the canonical Kobayashi hyperbolic direct limit there exists a holomorphic mapping $\eta\colon Z\to \Lambda$ such that for all $n\geq 0$ we have $\eta\circ \alpha_n= \lambda_n,$
that is $$\eta\circ \tau^{-n}\circ \ell= \v^{-n}\circ h.$$
Notice that this implies  $\eta\circ \ell=h$, and 
if $n\geq 0$, $$\v\circ\eta\circ \tau^{-1}\circ \alpha_n=\v\circ \v^{-n-1}\circ h=\lambda_n=\eta\circ\alpha_n.$$
Thus by the universal property of the canonical Kobayashi hyperbolic direct limit, $\eta=\v\circ\eta\circ \tau^{-1}$.
 Hence the mapping $\eta\colon Z\to \Lambda$ gives a morphism of semi-models $\hat\eta\colon  (Z,\ell,\tau)\to(\Lambda, h,\varphi )$.

Property (1)  follows clearly from Theorem \ref{cocononaut}. Property (1) implies in particular that for all $m\geq 0$ and $x\in X$,
the $m$-step $s_m(x)$ satisfies
$$s_m(x)=k_Z(\ell(z),\tau^m(\ell(z))).$$
By \cite[Proposition 2.7]{ABmod} 
$$c(\tau)=\lim_{m\to\infty}\frac{k_Z(\ell(z),\tau^m(\ell(z)))}{m}=\lim_{m\to\infty}\frac{s_m(x)}{m}=\lim_{m\to\infty}\frac{k_X(f^m(x),x)}{m}= c(f),$$
and
$$c(\tau)=\inf_{m\in \N}\frac{k_Z(\ell(z),\tau^m(\ell(z)))}{m}=\inf_{m\in \N}\frac{s_m(x)}{m},$$
which proves Property (2).

\end{proof}
\begin{remark}
Actually, the proof shows that the semi-model $(Z,\ell,\tau)$ satisfies the following stronger universal property. If $\Lambda$ is a Kobayashi hyperbolic complex manifold, if $\v\colon \Lambda\to \Lambda$ is an automorphism and if $h\colon X\to \Lambda$ is a holomorphic mapping such that $h\circ f=\v\circ h$ (notice that we do not assume (\ref{due})), then there exists a holomorphic mapping $\eta\colon Z\to \Lambda$ such that 
$\eta\circ \ell=h$ and $\eta\circ \tau= \v\circ\eta$.
Clearly, $\eta(Z)=\cup_{n\geq 0} \v^{-n}h(X)$.
\end{remark}

\section{The unit ball}\label{theunitball}

\begin{definition}\label{jafar}

The Siegel upper half-space $\mathbb H^q$ is defined by $$\mathbb{H}^q=\left\{(z,w)\in \C\times \C^{q-1}, \Im(z)>\|w\|^2\right\}.$$ Recall that $\mathbb H^q$ is biholomorphic to the ball $\B^q$ via the {\sl Cayley transform} $\Psi\colon \B^q\to \H^q$ defined as $$\Psi(z,w)=\left(i\frac{1+z}{1-z},\frac{w}{1-z}\right), \quad (z,w)\in \C\times \C^{q-1}.$$

Let $\langle\cdot, \cdot\rangle$ denote the standard Hermitian product in $\C^q$. In several complex variables, the natural generalization of the non-tangential limit at the boundary is the following.
 If $\zeta\in \partial\B^q$, then the set $$K(\zeta,R)\coloneqq\{z\in \B^q: |1-\langle z,\zeta\rangle|< R(1-\|z\|)\}$$ is a {\sl Kor\'anyi region} of {\sl vertex} $\zeta$ and {\sl amplitude} $R> 1$.
Let $f\colon \B^q \to \C^m$ be a holomorphic map. We say that $f$ has {\sl $K$-limit} $L\in \C^m$  at
$\zeta$ (we write $K\hbox{-}\lim_{z\to \zeta}f(z)=L$) if for
each sequence $(z_n)$ converging to $\zeta$ such that
$(z_n)$  belongs eventually to some Kor\'anyi region of vertex $\zeta$, we have
that $f(z_n)\to L$.
The Kor\'anyi regions can also be easily described in the Siegel upper half-space  $\mathbb H^q$, see e.g. \cite{BGP}.

Let $\zeta\in \partial \B^q$. 
A sequence $(z_n)\subset \B^q$ converging to $\zeta\in \partial\B^q$ is said to be {\sl restricted} at $\zeta$ if   $\la z_n, \zeta\ra\to 1$ non-tangentially in $\D$, while  it is said to be {\sl special} at $\zeta$ if
\[
\lim_{n\to \infty}k_{\B^q}(z_n,\langle z_n,\zeta\rangle \zeta)=0.
\]

We say that $f$ has {\sl  restricted $K$-limit}
$L$ at $\zeta$ (we write $\angle_K\lim_{z\to \zeta}f(z)=L$) if
for every special and restricted sequence $(z_n)$ converging to $\zeta$ we have
that $f(z_n)\to L$.
\end{definition}

One can show that
\[
K\hbox{-}\lim_{z\to \zeta}f(z)=L\Longrightarrow \angle_K\lim_{z\to
\zeta}f(z)=L,
\]
but the converse implication  is not true in general.

\begin{definition}\label{magamago'}
A point $\zeta\in \partial \B^q$ such that $K\hbox{-}\lim_{z\to\zeta}f(z)=\zeta$ and
$$\liminf_{z\to\zeta}\frac{1-\|f(z)\|}{1-\|z\|}=\lambda<+\infty$$ is called a {\sl boundary regular fixed point}, and $\lambda$ is called  its {\sl dilation}.

\end{definition}

The following  result \cite{H} generalizes the Denjoy--Wolff theorem in the unit disc.
\begin{theorem}
Let $f\colon \B^q\to \B^q$ be holomorphic. Assume that $f$ admits no fixed points in $\B^q$.
Then there exists a  point $p\in \de \B^q$, called the {\sl Denjoy--Wolff point} of $f$, such that $(f^n)$ converges uniformly on compact subsets to the constant map $z\mapsto p$. The Denjoy--Wolff point of $f$ is a boundary regular fixed point and 
its dilation $\lambda$ is smaller than or equal to $1$.
\end{theorem}

\begin{remark}\label{peterpan}
Let $f\colon\B^q\to\B^q$ be a holomorphic self-map  without fixed points, and let $\lambda$ be the dilation at its Denjoy--Wolff fixed point. Then by \cite[Proposition 5.8]{ABmod} the divergence rate of $f$ satisfyies $$c(f)=-\log \lambda.$$ 
\end{remark}

\begin{definition}
A holomorphic self-map $f\colon \B^q\to \B^q$ is called
\begin{enumerate}
\item {\sl elliptic} if it admits a fixed point $z\in \B^q$,
\item {\sl parabolic} if it admits no fixed points $z\in \B^q$, and its dilation at the Denjoy--Wolff point is  equal to $1$,
\item {\sl hyperbolic } if it admits no fixed points $z\in \B^q$, and its dilation at the Denjoy--Wolff point  is strictly smaller than $1$.  
\end{enumerate}
If $s_1(z)>0$ for all $z\in \B^q$, then we say that $f$ is {\sl  nonzero-step}.
\end{definition}

The next result generalizes Theorem \ref{blu} to the unit ball.

\begin{theorem}\label{fin1}
Let $f\colon \B^q\to \B^q$ be a hyperbolic holomorphic self-map, with dilation $\lambda$ at its Denjoy--Wolff point $p\in \partial \B^q$.
Then there exist
\begin{enumerate}
\item an integer $k$ such that $1\leq k\leq q$, 
\item a hyperbolic automorphism $\v\colon \H^k\to \H^k$ of the form
\begin{equation}\label{lomo1}
\v(z,w)= \left(\frac{1}{\lambda} z,\frac{e^{it_1}}{\sqrt \lambda}w_1,\dots, \frac{e^{it_{k-1}}}{\sqrt \lambda}w_{k-1} \right),
\end{equation}
where $t_j\in \R$ for $1\leq j\leq k-1$, 
\item a holomorphic mapping
$h\colon \B^q\to \H^k$ with  $K\hbox{-}\lim_{x\to p}h(x)=\infty,$
\end{enumerate}
 such that the triple
 $(\H^k,h,\v)$ is a  canonical Kobayashi hyperbolic model for $f$.
\end{theorem}
\begin{proof}
Since $\B^q$ is cocompact and Kobayashi hyperbolic, by Theorem \ref{principaleforward} there exists a
 canonical Kobayashi hyperbolic semi-model $(Z, \ell,\tau)$ for $f$.
Since $Z$ is a holomorphic retract of $\B^q$, it is biholomorphic to $\B^k$ for some $0\leq k\leq q$ (see e.g. \cite[Corollary 2.2.16]{A}). By Remark \ref{peterpan} and by (2) of Theorem \ref{principaleforward}, we  have  $c(\tau)=c(f)=-\log \lambda$,  hence $k\geq 1$ and $\tau$ is a hyperbolic automorphism  with dilation $\lambda$ at its Denjoy--Wolff point.
Thus there exists (see e.g. \cite[Proposition 2.2.11]{A}) a biholomorphism $\gamma\colon Z\to \H^k$ such that
$\varphi\coloneqq \gamma\circ\tau\circ\gamma^{-1} $ is of the form (\ref{lomo1}).
Setting $h\coloneqq  \gamma\circ\ell$ we have  that $(\H^k, h,\varphi)$ is also a canonical Kobayashi hyperbolic semi-model for $f$.
By \cite[Proposition 5.11]{ABmod}, we have $K\hbox{-}\lim_{x\to p}h(x)=\infty.$

\end{proof}

\begin{corollary}\label{valiron}
Let $f\colon \B^q\to \B^q$ be a hyperbolic holomorphic self-map, with dilation $\lambda$ at its Denjoy--Wolff point $p\in \partial \B^q$.
Then there exists a holomorphic mapping  $\vartheta\colon \B^q\to\H$ solving the Valiron equation (\ref{valironequation}).

\end{corollary}
\begin{proof}
Let $(\H^k,h,\v)$ be the  canonical Kobayashi hyperbolic semi-model given by Theorem \ref{fin1}. Let $\pi_1\colon \H^k\to \H$ be the projection $\pi_1(z,w)=z.$ Then $\left(\H,\vartheta\coloneqq \pi_1\circ h,x\mapsto \frac{1}{\lambda}x\right)$ is a semi-model for $f$, and thus $\vartheta$ solves the Valiron equation (\ref{valironequation}). 

\end{proof}

\begin{remark}
If $q=1$,  then the following uniqueness result holds \cite{BP}: any holomorphic solution of the Valiron equation (\ref{valironequation}) is a positive multiple of a given solution $\vartheta\colon \H\to \H$.

If $q\geq 2$, the situation is quite different. It is easy to see that the solutions of  (\ref{valironequation}) are all the holomorphic mappings of the form 
$ \Gamma \circ h,$ where $(\H^k,h,\v)$ is the  canonical Kobayashi hyperbolic semi-model given by Theorem \ref{fin1}, and  $\Gamma\colon \H^k\to \H$  is a holomorphic function such that   
\begin{equation}\label{stop}
\Gamma\circ \v=\frac{1}{\lambda}\Gamma.
\end{equation}
Notice that for all $z\in\H$, $$\Gamma\left(\frac{1}{\lambda}z,0\right)=\frac{1}{\lambda}\Gamma(z,0),$$
which by a result of Heins \cite{heins} implies that  $\Gamma(z,0)=az$ for some $a>0$ (and thus $\Gamma(\H^k)=\H$).
Thus if $k=1$ we obtain again a uniqueness result:  any holomorphic solution of (\ref{valironequation}) is a positive multiple of a given solution $\vartheta\colon \H^q\to \H$.

Assume now that $k\geq 2$. The function $\Gamma$ is  unique up to positive multiples on the slice $\{w=0\}$ of $\H^k$, but is far from being unique on  $\H^k\smallsetminus\{w=0\}$. This can be seen, for example, in the following way.
If $\gamma\colon \H^k\to \H^k$ is a holomorphic self-map which commutes with the hyperbolic automorphism $\v$, then clearly
$\Gamma\coloneqq \pi_1\circ \gamma$ satisfies (\ref{stop}). The family of holomorphic mappings of the form $\pi_1\circ \gamma$ is  large, as shown (and made precise) in \cite[Theorem 2.5]{defa}.
\end{remark}


Recall the following result on the Abel equation in the unit disc. 
\begin{theorem}[Pommerenke \cite{pommerenke}]
Let $f\colon \D\to \D$ be a parabolic nonzero-step holomorphic self-map. Then there exists a model 
$(\H,h, z\mapsto z\pm 1)$ for $f$.
\end{theorem}
The essential uniqueness of the intertwining mapping in the previous theorem is proved in \cite{PC3}.
The next result gives a generalization of this result to the unit ball.

\begin{theorem}\label{fin2}
Let $f\colon \B^q\to \B^q$ be a parabolic nonzero-step holomorphic self-map with  Denjoy--Wolff point $p\in \partial \B^q$.
Then there exist
\begin{enumerate}
\item an integer $k$ such that $1\leq k\leq q$, 
\item a parabolic automorphism $\v\colon \H^k\to \H^k$ of the form
\begin{equation}\label{lomo2}
\v(z,w)=(z\pm1,e^{it_1}w_1,\dots e^{it_{k-1}}w_{k-1}),
\end{equation}
where $t_j\in \R$ for $1\leq j\leq k-1$, or of the form
\begin{equation}\label{lomo3}
\v(z,w)=(z-2w_1+i, w_1-i,e^{it_2}w_2,\dots e^{it_{k-1}}w_{k-1}),
\end{equation}
 where where $t_j\in \R$ for $2\leq j\leq k-1,$
\item a holomorphic mapping
$h\colon \B^q\to \H^k$ with $\angle_K\hbox{-}\lim_{z\to 0}h(z)=\infty,$
\end{enumerate}
 such that the triple
 $(\H^k,h,\v)$ is a  canonical Kobayashi hyperbolic model for $f$.
\end{theorem}
\begin{proof}
Since $\B^q$ is cocompact and Kobayashi hyperbolic, by Theorem \ref{principaleforward} there exists a
 canonical Kobayashi hyperbolic semi-model $(Z, \ell,\tau)$ for $f$.
Since $Z$ is a holomorphic retract of $\B^q$, it is biholomorphic to $\B^k$ for some $0\leq k\leq q$.
Let $z\in Z$, $x\in \B^q$, and $n\geq 0$ such that $\tau^{-n}(\ell(x))=z$.
Then, by (1) of Theorem \ref{principaleforward}, $$k_Z(z,\tau(z))=s_1(z)>0.$$
Hence $k\geq 1$, and $\tau$ is not elliptic. By Remark \ref{peterpan} and by (2) of Theorem \ref{principaleforward}, we  have  $c(\tau)=c(f)=0$. Hence $\tau$ is parabolic.
 There exists (see e.g. \cite{defa-iann})  a biholomorphism $\gamma\colon Z\to \H^k$ such that
  $\varphi\coloneqq \gamma\circ\tau\circ\gamma^{-1} $
  is of the form (\ref{lomo2}) or of the form (\ref{lomo3}).
Setting $h\coloneqq  \gamma\circ\ell$ we have  that $(\H^k, h,\varphi)$ is also a canonical Kobayashi hyperbolic semi-model for $f$.
By \cite[Proposition 5.11]{ABmod}, we have $\angle_K\hbox{-}\lim_{x\to p}h(x)=\infty.$
\end{proof}

\part{Backward iteration}
\section{Preliminaries}
\begin{definition}
Let $X$ be a complex manifold.
We call   {\sl backward (non-autonomous) holomorphic dynamical system} on $X$ any
family $(f_{n,m}\colon X\to X)_{m\geq n\geq 0}$ of holomorphic self-maps such that for all $m\geq u\geq  n\geq 0$, we have
$$f_{n,u}\circ f_{u,m}=f_{n,m}.$$
 For all $n\geq 0$ we denote $f_{n,n+1}$ also by $f_n$.
A backward holomorphic dynamical system $(f_{n,m}\colon X\to X)_{m\geq n\geq 0}$ is called {\sl autonomous}
 if $f_{n}=f_{0}$ for all $n\geq 0$. Clearly in this case $f_{n,m}=f_0^{m-n}.$

\end{definition}
\begin{remark}
Any  family of holomorphic self-maps $(f_n\colon X\to X)_{n\geq 0}$ determines a   backward holomorphic dynamical system $(f_{n,m}\colon X\to X)$ in the following way: for all $n\geq 0$, set $f_{n,n}=\id$, and for all $m> n\geq 0$, set $$f_{n,m}=f_{n}\circ \cdots\circ f_{m-1}.$$ 
\end{remark}

\begin{definition}
Let $X$ be a complex manifold, and let $(f_{n,m}\colon X\to X)$ be a backward holomorphic dynamical system.
An {\sl inverse limit} for $(f_{n,m})$ is  a pair $(\Theta,V_n)$ where $\Theta$ is a set and
$(V_n\colon \Theta\to X)_{n\geq 0}$ is a family of  mappings such that 
 $$ f_{n,m}\circ V_m=V_n,\quad \forall \ m\geq n\geq 0,$$
 satisfying the following universal property:
if $Q$ is a set and if $(g_n\colon Q\to X)$ is a family of  mappings satisfying
$$ f_{n,m}\circ g_m=g_n,\quad \forall \ m\geq n\geq 0,$$
then there exists a  unique  mapping $\Gamma\colon Q\to \Theta$ such that
$$g_n=V_n\circ \Gamma,\quad \forall\ n\geq0.$$

\end{definition}
\begin{remark}
The inverse limit is essentially unique, in the following sense.
Let $(\Theta,V_n)$ and $(Q, g_n)$ be two inverse limits  for $(f_{n,m})$. Then there exists a bijective mapping 
$\Gamma\colon Q\to \Theta$ such that $$g_n= V_n\circ \Gamma,\quad \forall\ n\geq0.$$
\end{remark}

\begin{definition}
Let $X$ be a complex manifold, and let $(f_{n,m}\colon X\to X)$ be a backward holomorphic dynamical system.
A {\sl backward orbit} for $(f_{n,m})$ is a sequence $(x_n)_{n\geq 0}$ in $X$ such that, for all $m\geq n\geq 0$, $$f_{n,m}(x_m)=x_n.$$
\end{definition}
\begin{remark}
An inverse limit for $(f_{n,m})$ is easily constructed.
We define $\Theta$ as the set of all backward orbits for $(f_{n,m})$.
We define a family of mappings $(V_n\colon \Theta\to X)_{n\geq 0}$ in the following way. Let $\beta=(x_m)_{m\geq 0}$ be a backward orbit. Then for all $n\geq 0$, $$V_n(\beta)=x_n.$$
It is easy to see that $(\Theta, V_n)$ is an inverse limit for $(f_{n,m})$.
\end{remark}

\begin{definition}
Let $X$ be a complex manifold and let  $(f_{n,m}\colon X\to X)_{m\geq n\geq 0}$ be a  backward  holomorphic dynamical system.
Let $(\Theta, V_n)$ be the inverse limit of the inverse system $(X,f_{n,m})$. We define an equivalence relation $\sim$ on $\Theta$ in the following way. The backward orbits $(z_n)$ and $(w_n)$ are equivalent if and only if the non-decreasing sequence $(k_X(z_n,w_n))_{n\geq 0}$ is bounded. The class of the backward orbit $(z_n)$ will be denoted by $[z_n]$.
\end{definition}

\begin{lemma}
Let $X$ be a complex manifold, and let $(f_{n,m}\colon X\to X)$ be a backward  holomorphic dynamical system.
 Let $Z$ be a complex manifold and let
$(\alpha_n\colon Z\to X)$ be a sequence of holomorphic mappings such that 
$f_{n,m}\circ \alpha_m=\alpha_n$ for all $ m\geq n\geq 0.$ 
Then  $(\alpha_n(z)) \sim (\alpha_n(w))$ for all $z,w\in Z$.
\end{lemma}
\begin{proof}
It follows since $k_X(\alpha_n(z),\alpha_n(w))\leq k_Z(z,w)$ for all $n\geq 0.$
\end{proof}

We now introduce a modified version of the  inverse limit for $(f_{n,m})$ which is more suited for our needs.
\begin{definition}
Let $X$ be a complex manifold. Let $(f_{n,m}\colon X\to X)$ be a backward  holomorphic dynamical system.
We call  {\sl canonical inverse limit  associated with the class $[y_n]\in \Theta/_\sim$} for $(f_{n,m})$  a  pair $(Z,\alpha_n)$ where $Z$ is a complex manifold and
$(\alpha_n\colon Z\to X)$ is a sequence of holomorphic mappings such that 
\begin{enumerate}
\item $f_{n,m}\circ \alpha_m=\alpha_n,$ for all $ m\geq n\geq 0,$
\item $(\alpha_n(z))\in[y_n]$ for some (and hence for any) $z\in Z$,
\end{enumerate}
which satisfies the following universal property:
if $Q$ is a complex manifold and if $(g_n\colon Q\to X)$ is a family of holomorphic mappings satisfying
\begin{itemize}
\item[(1')] $f_{n,m}\circ g_m=g_n,$ for all $m\geq n\geq 0,$
\item[(2')] $(g_n(q))\in[y_n]$ for some (and hence for any) $q\in Q$, 
\end{itemize}
then there exists a unique holomorphic mapping $\Gamma\colon Q\to Z$ such that
$$g_n= \alpha_n\circ \Gamma,\quad \forall\ n\geq0.$$
\end{definition}

\begin{proposition}
The canonical inverse limit   for $(f_{n,m})$ associated with the class $[y_n]\in \Theta/_\sim$ is unique in the following sense.
Let $(Z,\alpha_n)$ and $(Q, g_n)$ be two canonical inverse limit for $(f_{n,m})$ associated with the same class $[y_n]$. Then there exists a biholomorphism 
$\Gamma\colon Q\to Z$ such that $$g_n= \alpha_n\circ \Gamma,\quad \forall\ n\geq0.$$
\end{proposition}
\begin{proof}
There exist holomorphic mappings $\Gamma\colon Q\to Z$ and $\Xi\colon Z\to Q$ such that for all $n\geq 0$, we have
$g_n=\alpha_n\circ \Gamma$ and $\alpha_n=g_n\circ \Xi$.
Thus the holomorphic mapping $\Gamma\circ \Xi\colon Z\to Z$ satisfies $$\alpha_n\circ \Gamma\circ \Xi=\alpha_n,\quad \forall\ n\geq0,$$
By the universal property of the canonical inverse limit  associated with the class $[y_n]\in \Theta/_\sim$, this implies that $\Gamma\circ \Xi=\id_Z$. Similarly, we obtain $\Xi\circ \Gamma=\id_Q$.
\end{proof}

\section{Non-autonomous iteration}
Let $X$ be a complete Kobayashi hyperbolic complex manifold.
Let $(f_{n,m}\colon X\to X)_{m\geq n\geq 0}$ be a backward holomorphic dynamical system, and assume that it admits   a relatively compact backward orbit $(y_m)_{m\geq 0}$. 
\begin{remark}\label{boco}
The class $[y_n]\in \Theta/_\sim$ coincides with the subset of $\Theta$ defined by all relatively compact backward orbits of $(f_{n,m})$.
\end{remark}

\begin{remark}\label{ancheB}
Let $K\subset X$ be a compact subset such that
$\{ y_m\}_{m\geq 0}\subset K.$
It follows that,  for all fixed $n\geq 0$,  
\begin{equation}\label{convergeB}
f_{n,m}(K)\cap K\neq \varnothing\quad \forall\, m\geq  n.
\end{equation}
\end{remark}

The sequence $(f_{0,m}\colon X\to X)_{m\geq 0}$ is not compactly divergent by (\ref{convergeB}), and since $X$ is taut, there exists a subsequence $(f_{0,m_{k_0}})_{k_0\geq 0}$ converging to a holomorphic self-map
$\alpha_0\colon X\to X$. The sequence $(f_{1,m_{k_0}}\colon X\to X)_{k_0\geq 0}$ is not compactly divergent by (\ref{convergeB}), and since $X$ is taut, there exists a subsequence $(f_{1,m_{k_1}})_{k_1\geq 0}$ converging to a holomorphic self-map $\alpha_1\colon X\to X$. Iterating this procedure we obtain a family of holomorphic self-maps
$(\alpha_n\colon X\to X)$ satisfying for all $m\geq n\geq 0$,
\begin{equation}\label{hermione}
f_{n,m}\circ \alpha_m=\alpha_n.
\end{equation}
Notice that for all $n\geq 0$ we have 
\begin{equation}\label{dumbledore}
\alpha_n(K)\cap K\neq \varnothing.
\end{equation}

Let now $(m_k)_{k\geq 0}$ be a sequence which for all $j\in \N$ is eventually a subsequence of $(m_{k_j})_{k_j\geq 0}$ (such a sequence exists by a diagonal argument).
The sequence of holomorphic self-maps $(\alpha_{m_k}\colon X\to X)_{k\geq 0}$
 is not compactly divergent by (\ref{dumbledore}), and since  $X$ is taut, there exists a subsequence $(\alpha_{m_h})_{h\geq 0}$
converging to a holomorphic self-map $\alpha\colon X\to X$.

\begin{lemma}
The holomorphic self-map $\alpha\colon X\to X$ is a holomorphic retraction, and for all $n\geq 0$,
\begin{equation}\label{potterB}
\alpha_n\circ \alpha=\alpha_n.
\end{equation}
\end{lemma}
\begin{proof}
Fix $n\geq 0$ and $x\in X$. Then for all $h\geq 0$ such that $m_h\geq n$, we have
$$\alpha_n(x)=f_{n,m_h}(\alpha_{m_h}(x))\stackrel{h\rightarrow\infty}\longrightarrow\alpha_n(\alpha(x)).$$
Thus we have, for all $h\geq 0$, $$\alpha_{m_h}( \alpha(x))=\alpha_{m_h}(x).$$
When $h\rightarrow \infty$, the left-hand side converges to $\alpha(\alpha(x))$, while the right-hand side converges to $\alpha(x)$.
\end{proof}
\begin{definition}
We denote the closed complex submanifold $\alpha(X)$ by $Z$.
\end{definition}

In what follows we denote the restriction $\alpha_n|_Z$ simply by $\alpha_n$.
Let $(\Theta,V_n)$ be the inverse limit of the inverse system $(X,f_{n,m})$. By the universal property of the inverse limit,
there exists a  mapping $\Psi\colon Z\to\Theta$ such that for all $n\geq 0$, $$\alpha_n=V_n\circ \Psi.$$
The mapping $\Psi$ is defined in the following way: if $z\in Z$, then $\Psi(z)$ is the backward orbit $(\alpha_m(z))_{m\geq 0}$.

\begin{proposition}\label{ipad}
The mapping $\Psi\colon Z\to \Theta$ is injective and its image is  $[y_n]$.
\end{proposition}
\begin{proof}
Let $z,w\in Z$ and assume  that $\Psi(z)=\Psi(w)$. It follows that $\alpha_m(z)=\alpha_m(w)$ for all $m\geq 0$, that is $\alpha(z)=\alpha(w)$. Since $\alpha$ is a retraction, we obtain $z=w$. Hence $\Psi\colon Z\to \Theta$ is injective.

We now  show that $\Psi(Z)\subset [y_n]$. If $z\in Z$, we have to show that the  sequence
$(k_X(\alpha_m(z),y_m))$ is bounded. Since $y_m\in K$ for all $m\geq 0$ and  $\alpha_{m_h}(z)\to \alpha(z)$, we have that the subsequence
$(k_X(\alpha_{m_h}(z),y_{m_h}))$ is bounded. Since the sequence $(k_X(\alpha_m(z),y_m))$ is  non-decreasing, it is bounded too.

Finally, we show that for all $(z_m)\in [y_n]$, there exists $z\in Z$ such that $\alpha_m(z)=z_m$ for all $m\geq 0$. Let thus $(z_m)$ be a backward orbit such that the sequence $(k_X(y_m,z_m))$ is bounded. 
 Clearly, the  subsequence $(k_X(y_{m_h},z_{m_h}))$ is also bounded, and thus there exists a subsequence  $(z_{m_u})$ of $(z_{m_h})$ converging to a point $z\in X$.
 It follows that for all $n\geq 0$, $$z_n=f_{n,m_u}(z_{m_u})\stackrel{u\rightarrow\infty}\to \alpha_n(z).$$
 We claim that $z\in Z$. Indeed, letting $u\rightarrow\infty$ in the identity $\alpha_{m_u}(z)=z_{m_u}$ we obtain $\alpha(z)=z$.
 \end{proof}

\begin{proposition}\label{3x}
The pair $(Z,\alpha_n)$ is a canonical inverse limit for $(f_{n,m})$ associated with $[y_n]$.
\end{proposition}

\begin{proof}
Let $Q$ be a complex manifold and let $(g_n\colon Q\to X)$ be a family of holomorphic mappings satisfying
\begin{enumerate}
\item $f_{n,m}\circ g_m=g_n,$ for all $m\geq n\geq 0,$
\item $(g_n(q))\in[y_n]$ for some (and hence for any) $q\in Q$. 
\end{enumerate}
By the universal property of the inverse limit, there exists a unique  mapping $\Phi\colon Q\to \Theta$ such that 
$$g_n=V_n\circ \Phi,\quad \forall\ n\geq0.$$
The mapping $\Phi$ is defined in the following way:  if $q\in Q$, then $\Phi(q)$ is the backward orbit $(g_m(q))_{m\geq 0}$.
 Property (2) implies that $\Phi(Q)\subset [y_n]$.
Set $$\Gamma\coloneqq\Psi^{-1}\circ \Phi\colon Q\to Z.$$
For all $n\geq 0$, 
\begin{equation}\label{gatto}
\alpha_n\circ \Gamma=V_n\circ \Psi\circ\Gamma=V_n\circ \Phi=g_n.
\end{equation}
The uniqueness of the mapping $\Gamma$ follows easily from the uniqueness of the mapping $\Phi$.
The mapping $\Gamma $ acts in the following way: if $q\in Q$, then $\Gamma(q)\in Z$ is uniquely defined by
\begin{equation}\label{mcgrannit}
\alpha_m(\Gamma(q))=g_m(q),\quad \forall\, m\geq 0.
\end{equation}

We now prove that $\Gamma$ is holomorphic. 
Recall that the sequence $(\alpha_{m_h}\colon Z\to X)_{h\geq 0}$ converges uniformly on compact subsets to $\id_Z$.
By Remark \ref{boco}, the sequence $(g_m\colon Q\to X )$ is not compactly divergent.
Since $X$ is taut, the sequence  $(g_{m_h}\colon Q\to X )$ admits a subsequence $(g_{m_u}\colon Q\to X )$ converging uniformly on compact subsets to a holomorphic mapping $g\colon Q\to X$. Thus taking the limit in both sides of 
$$\alpha_{m_u}\circ \Gamma=g_{m_u},
$$
as $m_u\to \infty$, we have $ \Gamma=g$, which implies that $\Gamma$ is holomorphic.
\end{proof}

\begin{proposition}\label{4x}
We have
$$\lim_{m\to\infty}\alpha_m^*k_X=k_Z,$$
and
$$\lim_{n\to\infty}\alpha_m^*\kappa_X=\kappa_Z.$$
\end{proposition}
\begin{proof}
Let $z,w\in Z$. We have $$\lim_{m_h\to\infty}k_X(\alpha_{m_h}(z),\alpha_{m_h}(w))=k_X(\alpha(z),\alpha(w))=k_X(z,w)=k_Z(z,w).$$
where the last identity follows from the fact that $Z$ is a holomorphic retract of $X$.
The first statement follows since the sequence $(k_X(\alpha_{m}(z),\alpha_{m}(w)))_{m\geq 0}$ is non-decreasing. 
The proof of the second statement is similar.
\end{proof}


\begin{theorem}\label{cocononautback}
Let $X$ a cocompact Kobayashi hyperbolic complex manifold, and let $(f_{n,m}\colon X\to X)_{m\geq n\geq 0}$ be a backward  dynamical system. Let $(y_n)$ be a backward orbit.
Then there exists a canonical inverse limit $(Z,\alpha_n)$  for $(f_{n,m})$ associated with $[y_n]$, where $Z$ is a holomorphic retract of $X$.
Moreover, 
\begin{equation}\label{gesso}
\lim_{m\to\infty}\alpha_m^*k_X=k_Z,\quad\mbox{and}\quad \lim_{m\to\infty}\alpha_m^*\kappa_X=\kappa_Z.
\end{equation}
\end{theorem}
\begin{proof}
Let $K\subset X$ be a compact subset such that $X=\Aut(X)\cdot K$.
For all $n\geq 0$, let $h_n\in\Aut(X)$ be such that $h_n^{-1}(y_n)\in K$. For all $m\geq n\geq 0$ set 
${\tilde f}_{n,m}=h_n^{-1}\circ f_{n,m}\circ h_m $. It is easy to see that $({\tilde f}_{n,m}\colon X\to X)$ is a forward holomorphic dynamical system with a relatively compact backward orbit $(\tilde y_n\coloneqq h_n^{-1}(y_n)).$
We can now apply Proposition \ref{3x} to $({\tilde f}_{n,m}\colon X\to X)$, obtaining a canonical  inverse limit $(Z,\tilde\alpha_n)$ for $({\tilde f}_{n,m})$ associated with $[\tilde y_n]$, where $Z$ is a holomorphic retract of $X$.
For all $n\geq0$ set $\alpha_n\coloneqq  h_n\circ\tilde \alpha_n$. Clearly $$ f_{n,m}\circ \alpha_m=\alpha_n,\quad \forall \ m\geq n\geq 0.$$
Let $Q$ be a complex manifold and let $(g_n\colon  Q\to X)$ be a family of holomorphic mappings satisfying
$$f_{n,m}\circ g_m=g_n,\quad \forall \ m\geq n\geq 0.$$
For all $n\geq0$ set $\tilde g_n\coloneqq  h_n^{-1}\circ g_n$.
Then for all $m\geq n\geq 0$, $${\tilde f}_{n,m}\circ \tilde g_m= {\tilde f}_{n,m}\circ h_m^{-1}\circ g_m=h_n^{-1}\circ f_{n,m}\circ g_m=\tilde g_n.$$
By the universal property of the canonical inverse limit  $(Z,\tilde\alpha_n)$ 
we obtain a holomorphic mapping $\Gamma\colon Q\to Z$ such that
$$\tilde g_n= \tilde\alpha_n\circ \Gamma,\quad \forall\ n\geq0.$$
Hence $g_n=\alpha_n \circ \Gamma$ for all $n\geq0$.

Finally, (\ref{gesso}) follows from Proposition \ref{4x}, since for all $n\geq 0$ the automorphism $h_n\colon X\to X$ is an isometry for $k_X$ and $\kappa_X$.

\end{proof}

\begin{remark}\label{lollo}
Let $(\Theta,V_n)$ be the inverse limit of the inverse system $(X,f_{n,m})$. Let $(y_n)$ be a backward orbit and let $(Z,\alpha_n)$ be the canonical inverse limit associated with $(y_n)$ given by Theorem \ref{cocononautback}. By the universal property of the inverse limit,
there exists a  mapping $\Psi\colon  Z\to \Theta$ such that  $$\alpha_n=V_n\circ \Psi,\quad \forall \ n\geq 0.$$
It is easy to see that $\Psi$ is injective and that $\Psi(Z)=[y_n]$.
In particular, for all $n\geq 0$, $$\alpha_n(Z)=V_n([y_n]).$$
\end{remark}

\section{Autonomous iteration}
\begin{definition}
Let $X$ be a complex manifold and let $f\colon X\to X$  be a holomorphic self-map. A {\sl pre-model} for $f$ is a triple $(\Lambda,h,\v)$ such that $\Lambda$ is a complex manifold, $h\colon \Lambda\to X$ is a holomorphic mapping and $\v\colon \Lambda\to \Lambda$ is an automorphism such that $$f\circ h=h\circ \v.$$ The mapping $h$ is called the {\sl intertwining mapping}. 

Let  $(\Lambda,h,\v)$ and $(Z,\ell,\tau)$  be two pre-models for $f$. A  {\sl morphism of pre-models} $\hat\eta\colon  (\Lambda,h,\v) \to (Z,\ell,\tau) $ is given by a holomorphic mapping  $\eta \colon \Lambda\to Z$ such that the following diagram commutes:
\SelectTips{xy}{12}
\[ \xymatrix{\Lambda \ar[rrr]^h\ar[rd]^\eta\ar[dd]^\v &&& X \ar[dd]^f\\
& Z \ar[rru]^\ell \ar[dd]^(.25)\tau\\
\Lambda\ar'[r][rrr]^(.25)h \ar[rd]^\eta &&& X\\
& Z \ar[rru]^\ell.}
\]
If the mapping $\eta \colon \Lambda\to Z$ is a biholomorphism, then we say that $\hat\eta\colon (\Lambda,h,\v) \to (Z,\ell,\tau) $ is an {\sl isomorphism of pre-models}. Notice that then $\eta^{-1}\colon Z\to \Lambda$ induces a morphism $ {\hat\eta}^{-1}\colon  (Z,\ell,\tau)\to   (\Lambda,h,\v). $
\end{definition}
\begin{definition}
Let $X$ be a complex manifold and let $f\colon X\to X$ be a holomorphic self-map.
Let $(y_n)$ be a backward orbit for $f$. 
 Let $(Z, \ell,\tau)$ be a semi-model for $f$ such that for some (and hence for any) $z\in Z$ we have $(\ell(\tau^{-n}(z)))\in [y_n]$.  
We say that  $(Z, \ell,\tau)$ is a  {\sl canonical pre-model associated with $[y_n]$}  for $f$   if for any pre-model
$(\Lambda,h,\v)$ for $f$ such that for some (and hence for any) $x\in \Lambda$ we have $(h(\v^{-n}(x)))\in [y_n]$,  there exists a unique morphism of pre-models $\hat\eta\colon (\Lambda,h,\v) \to (Z,\ell,\tau) $.
\end{definition}

\begin{remark}
If $(Z, \ell,\tau)$ and $(\Lambda,h,\v)$ are two canonical pre-models for $f$ associated with the same class $[y_n]$, then they are isomorphic.
\end{remark}

\begin{lemma}\label{nnamoacasa}
Let $X$ be a complex manifold and let $f\colon X\to X$  be a holomorphic self-map. 
Let $(y_n)$ be a backward orbit.  If there exists a canonical pre-model $(Z, \ell,\tau)$  for $f$ associated with $[y_n]$, then every backward orbit $(w_n)\in [y_n]$ has bounded step.
\end{lemma}
\begin{proof}
Let $z\in Z$. The backward orbit $(\ell(\tau^{-n}(z)))$ has bounded step since for all $n\geq 0$,
$$k_X(\ell(\tau^{-n}(z)),\ell(\tau^{-n-1}(z)))\leq k_Z(\tau^{-n}(z),\tau^{-n-1}(z))=k_Z(z,\tau(z)).$$
Let $(w_n)\in [y_n]$.
Since for all $n\geq 0$, $$k_{X}(w_n,w_{n+1})\leq k_{X}(w_n,\ell(\tau^{-n}(z)))+ k_{X}(\ell(\tau^{-n}(z)),\ell(\tau^{-n-1}(z)))+k_{X}(\ell(\tau^{-n-1}(z)),w_{n+1}),$$ it follows that $(w_n)$ has also bounded step.
\end{proof}

\begin{theorem}\label{principalebackward}
 Let $X$ be a cocompact  Kobayashi hyperbolic complex manifold, and let $f\colon X\to X$ be a holomorphic self-map. Let $(y_n)$ be a backward 
 orbit with bounded step.
 Then there exists a canonical  pre-model $(Z,\ell,\tau)$ for $f$ associated with  $[y_n]$, where $Z$ is a holomorphic retract of $X$.
Moreover, the following holds:
\begin{enumerate}
\item  $\ell(Z)=V_0([y_n])$, 
 \item if $\alpha_m\coloneqq\ell\circ\tau^{-m}$ for all $m\geq 0$, then
$$\lim_{m\to\infty}  \alpha_m^*\, k_X= k_Z,\quad
\lim_{m\to\infty} \alpha_m^*\, \kappa_X= \kappa_Z,$$
\item if $\beta$ is a backward orbit in the class $[y_n]$,
 $$c(\tau)=\lim_{m\to\infty}\frac{\sigma_m(\beta)}{m}=\inf_{m\in \N}\frac{\sigma_m(\beta)}{m}.$$
 \end{enumerate}
\end{theorem}
\begin{proof}
Let $(f_{n,m}\colon X\to X)$ be the autonomous dynamical system defined by $f_{n,m}=f^{m-n}$. By Theorem \ref{cocononautback}, there exist a holomorphic retract $Z$ of $X$ and a family of holomorphic mappings
$(\alpha_n\colon Z\to X)$ such that the pair $(Z,\alpha_n)$ is a canonical inverse limit associated with $[y_n]$.
The sequence of holomorphic mappings $(\beta_n\coloneqq f\circ\alpha_n\colon Z\to X)$ satisfies, for all $m\geq n\geq 0$,
$$f_{n,m}\circ \beta_m=f^{m-n}\circ f\circ \alpha_m=f\circ\alpha_n=\beta_n.$$
Let $z\in Z$ be the unique point such that $\alpha_m(z)=y_m$ for all $m\geq 0$.
Then for all $m\geq 1$,  $$k_X(\beta_m(z),y_m)=k_X(\alpha_{m-1}(z),y_m)=k_X(y_{m-1},y_m),$$
which is bounded since by assumption the backward orbit $(y_n)$ has bounded step.
By the universal property of the canonical inverse limit associated with $[y_n]$  there exists a  holomorphic self-map $\tau\colon Z\to Z$ such that for all $n\geq 0$, 
 $$\alpha_n\circ \tau= f\circ \alpha_n.$$

We claim that $\tau$ is a holomorphic automorphism.
Set  for all $n\geq 0$, $\gamma_n\coloneqq \alpha_{n+1}$. For all $m\geq n\geq 0$,
$$f_{n,m}\circ \gamma_m= f^{m-n}\circ \alpha_{m+1}=\alpha_{n+1}=\gamma_n.$$
Let $z\in Z$ be the unique point such that $\alpha_m(z)=y_m$ for all $m\geq 0$.
For all $m\geq 0$, $$k_X(\gamma_m(z),y_m)=k_X(\alpha_{m+1}(z),y_m)=k_X(y_{m+1},y_m),$$
which is bounded since by assumption the backward orbit $(y_n)$ has bounded step.
Thus there exists a holomorphic  self-map $\delta\colon Z\to Z$ such that $ \alpha_n\circ \delta=\alpha_{n+1}$ for all $n\geq 0$.
For all $n\geq 0$ we have $$\alpha_n\circ\tau\circ\delta= f\circ  \alpha_n\circ \delta=f\circ \alpha_{n+1}=\alpha_n,$$ and
$$\alpha_n\circ\delta\circ \tau=\alpha_{n+1}\circ \tau=\alpha_n.$$
  By the universal property of the canonical inverse limit associated with $[y_n]$ we have that $\tau $ is a holomorphic automorphism and $\delta=\tau^{-1}$.
Since for all $n\geq 0$, $$\alpha_n\circ \tau^n=f^n\circ \alpha_n=\alpha_0,$$
it follows that $$\alpha_n=\alpha_0\circ \tau^{-n}.$$

Set $\ell\coloneqq \alpha_0$. We claim that the triple $(Z,\ell, \tau)$ is a canonical pre-model for $f$ associated with $[y_n]$.
Indeed, let $(\Lambda, h,\varphi )$ be a pre-model for $f$ 
 such that for some (and hence for any) $x\in \Lambda$ we have $h(\v^{-n}(x))\in [y_n]$.
For all $n\geq 0$, let $\lambda_n\coloneqq h\circ \v^{-n}$. Then by the universal property of the canonical inverse limit associated with $[y_n]$  there exists a holomorphic mapping $\eta\colon  \Lambda\to Z$ such that for all $n\geq 0$ we have $\alpha_n\circ \eta= \lambda_n,$
that is $$\ell\circ \tau^{-n}\circ \eta=h\circ\v^{-n}.$$
Notice that this implies  $\ell\circ \eta=h$, and 
if $n\geq 0$, $$\alpha_n\circ\tau^{-1} \circ\eta \circ \v=h\circ  \v^{-n-1}\circ \v=\lambda_n.$$
Thus by the universal property of the canonical Kobayashi hyperbolic direct limit, $\eta= \tau^{-1}\circ\eta\circ \v.$
 Hence the mapping $\eta\colon \Lambda\to Z$ gives a morphism of pre-models $\hat\eta\colon (\Lambda, h,\varphi )\to (Z,\ell,\tau)$.

Property (1) follows from Remark \ref{lollo}. 
Property (2) follows from (\ref{gesso}). We now prove Property (3).
Let $\beta\coloneqq(w_n)$ be a backward orbit $[y_n]$, and let $z\in Z$ be the unique point such that $\alpha_n(z)=w_n$ for all $n\geq 0$.
Then by Property (2) the backward $m$-step $\sigma_m(\beta)$ satisfies
$$\sigma_m(\beta)=\lim_{n\to\infty}k_X(\alpha_n(z),\alpha_{n+m}(z))=\lim_{n\to\infty}k_X(\alpha_n(z),\alpha_{n}(\tau^{-m}(z)))=k_Z(z,\tau^{-m}(z)).$$
Notice that $k_Z(z,\tau^{-m}(z))=k_Z(z,\tau^{m}(z))$. We have
$$c(\tau)=\lim_{m\to\infty}\frac{k_Z(z,\tau^m(z))}{m}=\lim_{m\to\infty}\frac{\sigma_m(\beta)}{m},$$ and
$$c(\tau)=\inf_{m\in \N}\frac{k_Z(z,\tau^m(z))}{m}=\inf_{m\in \N}\frac{\sigma_m(\beta)}{m}.$$
\end{proof}

\section{The unit ball}\label{theunitballback}

\begin{definition}
Let $f\colon \B^q\to \B^q$ be a holomorphic self-map. Let $\zeta\in \partial \B^q$ be a boundary regular fixed point. The {\sl stable subset} of $f$ at $\zeta$  is defined as 
 the subset consisting of all $z\in \B^q$ such that there exists a backward orbit with bounded step starting at $z$ and  converging to $\zeta$. We denote it by $\mathcal S(\zeta)$.
\end{definition}
Clearly  $\mathcal S(\zeta)$ coincides with the union of all backward orbits in $\B^q$ with bounded  step converging to $\zeta$.

\begin{definition}
Let $f\colon \B^q\to \B^q$ be a holomorphic self-map. A {\sl boundary repelling fixed point} $\zeta\in \partial \B^q$ is a boundary regular fixed point  with dilation   $\lambda>1$.
\end{definition}

The next result generalizes Theorem \ref{bli} to the unit ball.
\begin{theorem}\label{fin3}
Let $f\colon \B^q\to\B^q$ be a holomorphic self-map and let $\zeta\in \partial\B^q$ be a boundary repelling fixed point  with dilation $1<\lambda<\infty$.  Let $(y_n)$ be a backward orbit with bounded step which converges to $\zeta$.
Define $\mu$ by  $$\mu\coloneqq {\lim_{m\to \infty}e^\frac{\sigma_m(\beta)}{m}}\geq \lambda,$$ where  $\beta\in[y_n]$. 
Then $\mu$ does not depend on $\beta\in[y_n]$ and there exist
\begin{enumerate}
\item an integer $k$ such that $1\leq k\leq q$, 
\item a hyperbolic automorphism $\v\colon \H^k\to\H^k$ with dilation $\mu$ at its unique repelling point $\infty$, of the form 
\begin{equation}\label{lomo4}
\v(z,w)= \left(\frac{1}{\mu} z,\frac{e^{it_1}}{\sqrt \mu}w_1,\dots, \frac{e^{it_{k-1}}}{\sqrt \mu}w_{k-1} \right),
\end{equation}
where $t_j\in \R$ for $1\leq j\leq k-1$,
\item a holomorphic mapping $h\colon  \H^k\to \B^q$ with  $K\hbox{-}\lim_{z\to \infty}h(z)=\zeta,$
\end{enumerate}
such that 
  $(\H^{k},h,\v)$ is a canonical pre-model for $f$ associated with $[y_n]$, 
  and $$h(\H^k)=V_0([y_n])\subset \mathcal{S}(\zeta).$$
  If $[y_n]$ contains backward orbit whose convergence to $\zeta$ is special and restricted, then $\mu=\lambda$.
\end{theorem}
\begin{proof}

Since $\B^q$ is cocompact and Kobayashi hyperbolic, by Theorem \ref{principalebackward} there exists a
 canonical pre-model $(Z, \ell,\tau)$ for $f$ associated with $[y_n]$.
Since $Z$ is a holomorphic retract of $\B^q$, it is biholomorphic to $\B^k$ for some $0\leq k\leq q$. 
By (3) of Theorem \ref{principalebackward},  if $\beta$ is a backward orbit in the class $[y_n]$,
 $$\mu=\lim_{m\to\infty}e^\frac{\sigma_m(\beta)}{m}=e^{c(\tau)}.$$
In particular, $\mu$ does not depend on $\beta\in[y_n]$. 

We claim that $\mu\geq\lambda$. Let $n\geq 0$.
Since $\lambda^n$ is the dilation at $\zeta$ of the mapping $f^n$, we have, for any $w\in \B^q$ (see e.g. \cite{A}),
$$n\log \lambda=\liminf_{z\to \zeta}(k_{\B^q}(w,z)-k_{\B^q}(w,f^n(z))).$$
Since $$k_{\B^q}(w,z)-k_{\B^q}(w,f^n(z))\leq k_{\B^q}(z,f^n(z)),$$ we have that 
 $n\log\lambda\leq {\sigma_n(\beta)}$, that is,
$\lambda\leq e^{\frac{\sigma_n(\beta)}{n}}$. Thus $\mu\geq\lambda$.

The automorphism $\tau$ is hyperbolic since the dilation  at its Denjoy--Wolff point is equal to $e^{-c(\tau)}$ and $$e^{-c(\tau)}=\frac{1}{\mu}\leq\frac{1}{\lambda}<1.$$

There exists (see e.g.  \cite[Proposition 2.2.11]{A}) a biholomorphism $\gamma\colon Z\to \H^k$ such that
$\varphi\coloneqq \gamma\circ\tau\circ\gamma^{-1} $ is of the form (\ref{lomo4}).
Setting $h\coloneqq  \ell\circ \gamma^{-1}$ we have  that $(\H^k, h,\varphi)$ is also a canonical pre-model for $f$ associated with $[y_n]$.

We now address the regularity at $\infty$ of the intertwining mapping $h$. Let $(z_n,w_n)$ be a backward orbit in $\H^k$ for $\tau$. Then  $(z_n,w_n)$ converges to $\infty$ and there exists $C>0$ such that  
$$k_{\H^k}((z_n,w_n),( z_{n+1},w_{n+1}))\leq C,\quad \mbox{and}\quad k_{\H^k}((z_n,w_n),(z_n, 0))\leq C.$$ 
 Clearly $g(z_n,w_n)$ is a backward orbit for $f$ which converges to $\zeta\in\partial\B^q$. Then  \cite[Theorem 5.6]{ABmod}  yields the result.

Theorem \ref{principalebackward} yields that $h(\H^k)=V_0([y_n])$. Let $x\in  V_0([y_n])$. Then there exists a backward orbit $(w_n)\in[y_n] $ starting at $x$, which clearly converges to $\zeta$. By Lemma \ref{nnamoacasa} the backward orbit $(w_n)$ has bounded step, and thus $ V_0([y_n])\subset \mathcal{S}(\zeta)$.

Let $\beta\coloneqq(w_n)$ be a  special and restricted backward orbit in $[y_n]$ converging to $\zeta$. Then the same proof as in \cite[Proposition 4.12]{Ar} shows that
$$\log \mu=\lim_{m\to\infty}\frac{\sigma_m(\beta)}{m}=\log\lambda.$$
\end{proof}

We leave the following open questions.
\begin{question}
With notations from the previous theorem, does the identity  $\lambda=\mu$ always hold?
\end{question}
\begin{question}
Let $f\colon \B^q\to\B^q$ be a holomorphic self-map and let $\zeta\in \partial\B^q$ be a boundary repelling fixed point  with dilation $1<\lambda<\infty$. By \cite[Lemma 3.1]{O},  if $\zeta$ is isolated
  from other boundary repelling fixed points with dilation less or equal than $\lambda$, then $\mathcal{S}(\zeta)\neq \varnothing$. Is the same true if the point $\zeta$ is not isolated?
\end{question}
\begin{question}
Let $f\colon \B^q\to\B^q$ be a parabolic self-map and let $p\in \partial\B^q$ be its Denjoy--Wolff point.  Let $(y_n)$ be a backward orbit with bounded step which converges to $p$. Let $(Z,\ell,\tau)$ be a canonical pre-model associated with $[y_n]$. Clearly $\tau$ cannot be elliptic. Is $\tau$ parabolic? 
In the unit disc, it follows from \cite[Theorem 1.12]{PC2} that this is true.
\end{question}

\end{document}